\documentclass[11pt]{amsart}

\usepackage[top=25mm, bottom=25mm, left=25mm, right=25mm]{geometry}  

\usepackage{mathpazo}
\usepackage{amsmath,amssymb, amscd, color}
\usepackage{mathtools}
\usepackage{graphicx}
\usepackage{import}
\usepackage{amscd}
\usepackage{wrapfig}
\usepackage{epsfig}
\usepackage{enumerate}
\numberwithin{equation}{section}
\usepackage[alphabetic,nobysame]{amsrefs}
\usepackage{float}
\usepackage{tikz}
\usetikzlibrary{positioning}
\usepackage{comment}
\usepackage{hyperref}
\DeclareMathOperator{\fl}{flat}
\DeclareMathOperator{\hyp}{hyp}
\DeclareMathOperator{\sol}{solv}
\DeclareMathOperator{\shadow}{Sh}
\DeclareMathOperator{\lad}{Lad}
\DeclareMathOperator{\slope}{Slope}

\DeclareMathOperator{\lio}{Lio}
\DeclareMathOperator{\tot}{total}

\newcounter{sebcomments}

\newcounter{vaibhavcomments}

\usepackage{cleveref}

\newtheorem{theorem}{Theorem}[section]

\newtheorem{lemma}[theorem]{Lemma}
\newtheorem{corollary}[theorem]{Corollary}

\newtheorem{definition}[theorem]{Definition}

\newtheorem{remark}[theorem]{Remark}

\makeatletter
 \let\c@theorem=\c@subsection
 \let\c@conjecture=\c@subsection
 \let\c@lemma=\c@subsection
 \let\c@proposition=\c@subsection
 \let\c@claim=\c@subsection
 \let\c@question=\c@subsection
 \let\c@criterion=\c@subsection
 \let\c@vfconj=\c@subsection
 \let\c@definition=\c@subsection
 \let\c@notation=\c@subsection
 \let\c@remark=\c@subsection
 \let\c@example=\c@subsection
 \let\c@equation=\c@subsection
 \let\c@figure=\c@subsection
 \let\c@wrapfigure=\c@subsection

\makeatother

\begin{document}
\title{Linear progress in fibres}

\author[Gadre]{Vaibhav Gadre}
\address{\hskip-\parindent
     School of Mathematics and Statistics\\
     University of Glasgow\\
     University Place\\
      Glasgow\\
      G12 8QQ United Kingdom}
\email{Vaibhav.Gadre@glasgow.ac.uk}
\thanks{The second author is supported in part by the DFG grant HE 7523/1-1 within the SPP 2026 ``Geometry at Infinity''}

\author[Hensel]{Sebastian Hensel}
\address{\hskip-\parindent
  Mathematisches Institut der Universität M\"unchen\\
  Theresienstr. 39\\
  D-80333 M\"unchen\\
  Germany
  }
\email{hensel@math.lmu.de}
%\thanks{}
 
%\keywords{\teichmuller theory, Moduli of Riemann surfaces.}
%\subjclass[2010]{30F60, 32G15}

%%%%%%%%%%%%%%%%%%%%%%%%%%%%%%%%%%%%%%%%%%%%%%%%

\begin{abstract}
A fibered hyperbolic 3-manifold induces a map from the hyperbolic plane to hyperbolic 3-space, the respective universal covers of the fibre and the manifold.
The induced map is an embedding that is exponentially distorted in terms of the individual metrics. 
In this article, we begin a study of the distortion along typical rays in the fibre. 
We verify that a typical ray in the hyperbolic plane makes linear progress in the ambient metric in hyperbolic 3-space.
We formulate the proof in terms of some soft aspects of the geometry and basic ergodic theory.
This enables us to extend the result to analogous contexts that correspond to certain extensions of closed surface groups. 
These include surface group extensions that are Gromov hyperbolic, the universal curve over a Teichm\"{u}ller disc, and the extension induced by the Birman exact sequence.
\end{abstract}

%%%%%%%%%%%%%%%%%%%%%%%%%%%%%%%%%%%%%%%%%%%%%%%%

\maketitle

%%%%%%%%%%%%%%%%%%%%%%%%%%%%%%%%%%%%%%%%%%%%%%%%
\section{Introduction}
%%%%%%%%%%%%%%%%%%%%%%%%%%%%%%%%%%%%%%%%%%%%%%%%

In this article, we initiate the study of \emph{distortion} for \emph{typical elements} in groups focusing on examples whose motivation comes from geometry and topology in dimensions 2 and 3.

\smallskip
Suppose that $H$ is a finitely generated subgroup of a finitely generated group $G$. 
For any choice of proper word metrics on $H$ and $G$, the inclusion of $H$ into $G$ is a Lipschitz map.
However, distances are contracted by arbitrary amounts by the inclusion in many examples.  
This can be quantified by the \emph{distortion function}, the smallest function bounding the word norm of $H$ in terms of the word norm of $G$.
Many examples from low-dimensional topology exhibit exponential distortion; well known examples are the fundamental groups of fibres in fibered hyperbolic 3-manifolds, Torelli and handlebody groups in surface mapping class groups.

\smallskip 
By definition, the distortion function measures the worst case discrepancy between intrinsic and ambient metrics -- for example, the existence of a single sequence of group elements in $H$ whose norm grows linearly in $G$ and exponentially in $H$ already implies exponential distortion.
 
 \smallskip
In this article, we adopt instead a more probabilistic viewpoint to ask about the growth of the ambient norm for \emph{typical} elements of the subgroup.

\smallskip 
Concretely, we consider subgroups isomorphic to the fundamental group of a closed surface and thus quasi-isometric to the hyperbolic plane. 
This allows us to use the (Lebesgue) measure on the circle at infinity to sample geodesics in the subgroup. 
Our main result proves that in various topologically interesting, exponentially distorted examples of surface group extensions, such paths nevertheless make linear progress in the ambient space. 
We consider three geometrically motivated contexts, the first of which from geometric group theory.
For the entirety of our article, we let $\Sigma$ be a closed orientable surface of genus $g \geqslant 2$. 

\begin{theorem}\label{thm:intro-extension}
	Let
	 \[ 1 \to \pi_1(\Sigma) \to \Gamma \to Q \to 1 \] 
	 be a hyperbolic group extension of a closed surface group. Then there is a constant $c>0$ that depends only on the word metrics, so that almost every geodesic ray in $\pi_1(\Sigma)$ (sampled by the Lebesgue measure under an identification of $\pi_1(\Sigma)$ with $\mathbb{H}^2$) makes linear progress with speed at least $c$ in the word metric on $\Gamma$.
\end{theorem}
In particular, this includes the classical case where $\Gamma$ is a
fundamental group of a fibered hyperbolic $3$--manifold. This special
case is likely well-known to experts (but to our knowledge not
published). We also want to mention here that all known groups of this
type have $Q$ virtually free (but this does not yield any
simplifications for the question considered here). 

\smallskip 
The second context arises in Teichm\"uller theory. 
A holomorphic quadratic differential on a Riemann surface is equivalent to a collection of charts from the surface to $\mathbb{C}$ with transition functions that are half-translations, that is of the form $z \to \pm z + c$. 
The $SL(2, \mathbb{R})$-action on $\mathbb{C} = \mathbb{R}^2$ preservers the form of the transitions and hence descends to an action on the space of quadratic differentials. 
The compact part $SO(2,\mathbb{R})$ acts by rotations on the charts and hence preserves the underlying conformal structure on the surface. 
Thus, given a quadratic differential $q$, the $SL(2, \mathbb{R})$ orbit of $q$ gives an isometric embedding of $\mathbb{H}^2 = SL(2,\mathbb{R}) \, / SO(2,\mathbb{R})$ in the Teichm\"uller space of the surface. This is called the associated Teichm\"uller disk $D_q$. 
We may then consider a bundle $E \to D_q$ whose fibres are the universal covers of the corresponding singular flat surfaces.
The bundle $E$ carries a natural metric, in which the fibres are again exponentially distorted. 
Although the total space is not hyperbolic, we obtain here:

\begin{theorem}\label{thm:intro-tm}
For any quadratic differential $q$, there is a number $c>0$, so that any geodesic in a fibre of $E$ (sampled with Lebesgue measure under an identification with $\mathbb{H}^2$) makes linear progress with speed at least $c$ in the metric on $E$.
\end{theorem}

Finally, we consider the Birman exact sequence 
\[ 
1 \to \pi_1(\Sigma, p) \to \mathrm{Mod}(\Sigma-p) \to \mathrm{Mod}(\Sigma) \to 1, 
\]
where $\mathrm{Mod}(\Sigma)$ is the mapping class group of a closed orientable surface $\Sigma$ of genus $g \geqslant 2$ and $\mathrm{Mod}(\Sigma - p)$ is the mapping class group of the surface $\Sigma$ punctured/ marked at the point $p$. 

\smallskip
Here, we obtain
\begin{theorem}\label{thm:intro-birman}
	For the Birman exact sequence
	\[ 1 \to \pi_1(\Sigma, p) \to \mathrm{Mod}(\Sigma - p) \to \mathrm{Mod}(\Sigma) \to 1, \]
	there is a number $c>0$, so that any geodesic in $\pi_1(\Sigma , p)$ (sampled with Lebesgue measure under an identification with $\mathbb{H}^2$) makes linear progress with speed at least $c$ in a word metric on $\mathrm{Mod}(\Sigma-p)$.
\end{theorem}
The three results are not unexpected.
%In particular, we think that \Cref{thm:intro-extension} in all likelihood is well-known to experts, but to our knowledge not published.
The main merit of our article is that we distil the key features, thus giving a unified treatment in all three contexts that differ substantially in their details.
Our results also leave the topic poised for a finer exploration of distortion in all contexts.

\subsection*{Proof Strategy}

%\smallskip
The basic method of proof is the same for all three results, and has two main parts. 
In the geometric part, we construct suitable shadow-like sets in the total space from a ladder-type construction motivated by the construction in \cite{Mit}. 
The basic idea is to move shadows in the base fibre to all other fibres using the monodromy and consider their union.
We implement the idea with sufficient care to ensure that for a pair of nested shadows in the base fibre the shadow-like set from the bigger shadow is contained in the shadow-like set from the smaller shadow.

\smallskip
We then characterise \texttt{"}good\texttt{"} geodesic segments in the fibre, that is segments for which the shadow-like set at the end of the segment is nested in the shadow-like set at the beginning by a distance in the ambient space that is linear in the length of the segment. 
This translates into a progress certificate for fibre geodesics in the ambient space.

\smallskip
Next, in the dynamical part we exploit ergodicity of the geodesic flow on hyperbolic surfaces to guarantee that fellow-travelling with good segments occurs with a positive asymptotic frequency thus proving the results.

\smallskip 
The argument is cleanest in the classical case of a fibered hyperbolic $3$--manifold. 
In this case, the monodromy is by a pseudo-Anosov map of the fibre surface.
Such a map has a unique invariant Teichm\"uller axis and is a translation along the axis.
By Teichm\"uller's theorem, the surface can be equipped with a quadratic differential such that the map is represented by an affine map (given by a diagonal matrix in $SL(2,\mathbb{R})$) in the singular flat metric on the surface defined by the quadratic differential. 
We can thus identify the fibre group by a quasi-isometry with universal cover of the surface with the lifted singular flat metric.
The pseudo-Anosov monodromy acting as an affine map defines the singular flat metrics on the other fibres. 
Although the monodromy does not act as an isometry of the singular flat metric, it maps geodesics to geodesics.
A long straight arc, for example a long saddle connection that makes an angle close to $\pi/4$ with both the horizontal and vertical foliations of $q$ has the property that all its images under the pseudo-Anosov monodromy will be also be long. 
Using (Gromov) hyperbolicity, the property of containing/ fellow-travelling such a segment is stable for geodesic rays.
In particular, it has positive Liouville measure when we pass to the hyperbolic metric. 
Hence, ergodicity of the hyperbolic geodesic flow implies that a typical hyperbolic ray satisfies the property with a positive asymptotic frequency.
This ensures linear progress.

\smallskip
In a general hyperbolic surface group extension, the monodromies have weaker properties. 
In the bundle over a Teichm\"uller disc, the ambient space is not hyperbolic but the monodromies are affine.
In the Birman exact sequence, both properties fail but weaker ones hold which still suffice to run our strategy.
We work around these problems and defining good segments that they give robust distance lower bounds in the ambient space requires the main care.

\subsection*{Other sampling methods}
In Theorems~\ref{thm:intro-extension} and~\ref{thm:intro-birman}, one could also sample in the kernel subgroup using random walks. 
Here, general results on random walks in hyperbolic groups ensure that corresponding results are true as well. 
In light of the famous Guivarc'h--Kaimanovich--Ledrappier singularity conjecture (see \cite[Conjecture 1.21]{Der-Kle-Nav}) for stationary measures, the random walk sampling is in theory different from the earlier sampling using the hyperbolic Liouville measure.

\smallskip
In the setup of Theorem~\ref{thm:intro-tm}, there is no direct random walk analog. 
In the lattice case, that is when the affine symmetric group $SL(X, q)$ (sometimes known as the Veech group) of a quadratic differential $q$ is a lattice in $SL(2,\mathbb{R})$, one may replace the bundle $E$ by the extension of $SL(X,q)$. 
Recent work of \cite{DDLS} shows that extension group acts on a suitable hyperbolic space, in fact, the result is hierarchically hyperbolic. 
This in turn again implies a linear progress result for sampling using stationary measures for random walks on $SL(X, q)$.

\subsection*{Future Directions}
Our results lay the preliminary ground work for more refined questions regarding distortion statistics for a random sampling in the contexts we consider.
We will outline one such direction here.

\smallskip
Organising distortions by scale we may ask for the explicit statistics of distortion along typical geodesics.
An example of this nature covered in the literature is the case of a non-uniform lattice in $SL(2,\mathbb{R})$, more generally a non-uniform lattice in $\mathrm{Isom}(H^n)$. 
Because of the presence of parabolic elements, a non-uniform lattice in $SL(2,\mathbb{R})$ is distorted in $\mathbb{H}^2$ under the orbit map.
The orbit is confined entirely to some thick part and never enters horoballs that project to cusp neighbourhoods in the quotient hyperbolic surface. 
The hyperbolic geodesic segment that joins pairs of orbit points that differ by a power of a parabolic is logarithmic in this power.
This gives rise to the exponential distortion in this context.
For example, when the lattice is $SL(2,\mathbb{Z})$ the distortion statistics can be interpreted as the statistics of continued fraction coefficients.
Similar distortion coming from a \texttt{"}parabolic\texttt{"} source arises in mapping class groups.
To see some analysis of the distortion statistics in such examples we refer the reader to \cite{Gad-Mah-Tio1}, \cite{Gad-Mah-Tio2} and \cite{Ran-Tio}.

\smallskip
In the contexts we consider here, the distortion does not have a parabolic source and results analogous to the examples in the above paragraph would be quite interesting.

%%%%%%%%%%%%%%%%%%%%%%%%%%%%%%%%%%%%%%%%%%%%%%%%%%%%%
\section{Linear Progress in fibered hyperbolic 3-manifolds}\label{s.mainprogress}
%%%%%%%%%%%%%%%%%%%%%%%%%%%%%%%%%%%%%%%%%%%%%%%%%%%%%

Let $M$ be a closed hyperbolic 3-manifold that fibres over the circle with fibre a closed orientable surface $\Sigma$ with genus $g \geqslant 2$. 
Fixing a fibre $\Sigma$, the manifold $M$ can be realised as a mapping torus $\Sigma \times  [0,1]/ \sim$, where $\Sigma \times \{0\}$ has been identified with $\Sigma \times \{1 \} $ by a pseudo-Anosov monodromy $f$, that is, by a mapping class $f$ of $\Sigma$ that is pseudo-Anosov.

\smallskip
Passing to the universal covers, the inclusion of $\Sigma$ in $M$ as the fibre $\Sigma \times \{0 \}$ induces an inclusion of the universal cover $\mathbb{H}^2$ of the fibre $\Sigma$ to the universal cover $\mathbb{H}^3$ of $M$. 
In the hyperbolic metrics on $\mathbb{H}^2$ and $\mathbb{H}^3$, this inclusion is distorted. 
Nevertheless, Cannon-Thurston \cite{Can-Thu} proved that there is limiting behaviour at infinity despite distortion.
More precisely, they showed that the inclusion induces a continuous map from $S^1 = \partial_\infty \mathbb{H}^2$ to $S^2 = \partial_\infty \mathbb{H}^3$, and moreover this map is surjective.
Thus, it is implicit that the image in $\mathbb{H}^3$ of any hyperbolic geodesic ray $\gamma$ in $\mathbb{H}^2$ converges to a point in $S^2 = \partial_\infty \mathbb{H}^3$, even though the image need not be a quasi-geodesic because of the distortion.

\subsection{Sampling geodesics} 
Our first notion of sampling involves the hyperbolic geodesic flow on $T^1 \Sigma$. 
Let $g_t$ be the hyperbolic geodesic flow on $T^1 \Sigma$ and let $\mu_{\lio}$ be the $g_t$-invariant Liouville measure on $T^1 \Sigma$. 
Let $\pi: T^1 \Sigma\to \Sigma$ be the canonical projection.
We adopt the convention that when we mention a hyperbolic geodesic ray $\gamma$, we mean the projection $\pi(g_t v) \, : \, t \geqslant 0$ for some $v \in T^1 \Sigma$. 

\smallskip
Let $D_{\hyp}$ be the hyperbolic metric on $\mathbb{H}^3$. 
% Here, we prove
Sections~\ref{s.mainprogress} to~\ref{s.mainproof1} will be concerned with a detailed proof of the following theorem (which will also serve as a blueprint for the analogous result in other settings).

\begin{theorem}\label{t.linear-progress}
There exists a constant $k > 0$ such that for $\mu_{\lio}$-almost every $v \in T^1 \Sigma$, the corresponding hyperbolic geodesic ray $\gamma = \pi(g_t v)$ satisfies
\[
D_{\hyp} (\gamma_0, \gamma_T)  > k T
\]
for all $T$ sufficiently large depending on $v$. 
\end{theorem}

%In other words, we prove that the typical geodesic ray is undistorted \texttt{"}on average\texttt{"} and makes linear progress in the hyperbolic metric on $\mathbb{H}^3$. 

Before beginning with the proof in earnest, we also want to discuss a different method of sampling random geodesics in the fiber surface.

For this second notion, we consider non-elementary random walks on $\pi_1(\Sigma)$. 
Let $\mu$ be a probability distribution on $\pi_1(\Sigma)$. 
A sample path of length $n$ for a $\mu$-random walk on $\pi_1(\Sigma)$ is the random group element $w_n$ given by $w_n = g_1 g_2 \cdots g_n$, where each $g_j$ is sampled by $\mu$, independently of the preceding steps. 
A random walk is said to be \emph{non-elementary} if the semi-group generated by the support of $\mu$ contains a pair of hyperbolic elements with distinct stable and unstable fixed points on $S^1 = \partial_\infty \mathbb{H}^2$. 

\smallskip
It is a classical fact (generalised by Furstenberg to many more settings) that a non-elementary 
random walk on $\pi_1(\Sigma)$, when projected to $\mathbb{H}^2$ using the group action, converges to infinity. That is, for any base-point $x$ and for almost every infinite sample path $\omega = (w_n)$, the sequence $w_n x$ in $\mathbb{H}^2$ converges to a point of $S^1$. 
For a more detailed account of the theory, we refer the reader to \cite{Mah-Tio}.

\smallskip
The almost sure convergence to $S^1$ defines a stationary measure $\nu$ on it. 
We may then use $\nu$ to sample geodesic rays as follows. 
Using geodesic convergence to infinity, we can pull $\nu$ back to a measure on the unit tangent circle $T^1_x \mathbb{H}^2$.
We use this pull-back to sample hyperbolic geodesic rays starting from $x$ and consider the question of whether a typical ray makes linear progress.

\smallskip
A famous conjecture of Guivarc'h--Kaimanovich--Ledrappier (see \cite[Conjecture 1.21]{Der-Kle-Nav}) states that for any finitely supported non-elementary random walk on $\pi_1(\Sigma)$ the associated stationary measure $\nu$ is singular with respect to the Lebesgue measure on $S^1$. As such, linear progress of $\nu$-typical ray cannot be deduced from \Cref{t.linear-progress}.

\smallskip
For technical reasons, it is more convenient to simultaneously consider forward and backward random walks. 
The backward random walk is simply the random walk with respect to the reflected measure $\hat{\mu}(g) = \mu(g^{-1})$. 
The space of bi-infinite sample paths, denoted by $\Omega$, has a natural invertible map on it given by the right shift $\sigma$.
The product measure $\nu \times \hat{\nu}$, where $\hat{\nu}$ is the stationary measure for the reflected random walk, is $\sigma$-ergodic.

\smallskip
Using the orbit map, we may equip $\pi_1(\Sigma)$ with the hyperbolic metric induced from $\mathbb{H}^2$, that is, we may consider the functions $f_n (\omega) = d_{\hyp} (x, w_n x)$ along sample paths in $\Omega$.  
Similarly, by the orbit map to $\mathbb{H}^3$, we may also equip $\pi_1(\Sigma)$ with the function induced by the hyperbolic metric from $\mathbb{H}^3$, that is, we may consider the functions $F_n(\omega) = D_{\hyp} (x, w_n x)$ along sample paths in $\Omega$. 

\smallskip
A distribution $\mu$ on $\pi_1(\Sigma)$ is said to have finite first moment for a word metric if 
\[
\sum\limits_{g \in \pi_1(\Sigma)} d_{\pi_1(\Sigma)} (1, g) \, d \mu(g) < \infty.
\]
The action of $\pi_1(\Sigma)$ on $\mathbb{H}^2$ is co-compact.
Hence, a word metric on $\pi_1(\Sigma)$ is quasi-isometric to the hyperbolic metric $d_{\hyp}$ induced on it by the orbit map. 
Thus, finite first moment for a word metric is equivalent to finite first moment for $d_{\hyp}$ and so we no longer need to specify the metric.
It follows that if $\mu$ has finite first moment then the functions $f_n$ are $\ell^1$ with respect to $\nu \times \hat{\nu}$. 
Since $f_n \geqslant F_n$, we deduce that the functions $F_n$ are $\ell^1$. 
%Henceforth, we consider only random walks with finite first moment.

\smallskip
By the triangle inequality, the function sequences $f_n$ and $F_n$ are sub-additive along sample paths.
Hence, by Kingman's sub-additive ergodic theorem, there exists constants $a_1 \geqslant 0$ and $a_2 \geqslant 0$ such that for almost every bi-infinite sample path $\omega$ 
\[
\lim_{n \to \infty} \frac{f_n(\omega)}{n} = a_1 \hskip 10pt \text{and} \hskip 10pt  \lim_{n \to \infty} \frac{F_n(\omega)}{n} = a_2.
\]
The constants are called drifts. 

\smallskip
To argue that the drifts are positive, we recall \cite[Theorem 1.2]{Mah-Tio}. 
\begin{theorem}[Maher--Tiozzo]\label{t.Mah-Tio} 
Suppose that a countable group $G$ acts by isometries on a separable Gromov hyperbolic space $X$ and let $x$ be any point of $X$. Let $\mu$ be a non-elementary probability distribution on $G$, that is, the semigroup generated by the support of $\mu$ contains a pair of hyperbolic isometries of $X$ with distinct fixed points on the Gromov boundary $\partial X$. Further suppose that $\mu$ has finite first moment in $X$, that is, 
\[
\sum\limits_{g \in G} d_X (x, gx) \, d_\mu(g) < \infty.
\]
Then there is a constant $L >0$ such that for any base-point $x \in X$ and for almost every sample path $\omega = (w_n)_{n \in \mathbb{N}}$ 
\[
\lim_{n \to \infty} \frac{d_X(x, w_n x) }{n} = L.
\]
\end{theorem} 
In our case, the actions of $\pi_1(\Sigma)$ on $\mathbb{H}^2$ and $\mathbb{H}^3$ are both non-elementary. Since $D_{\hyp} \leqslant d_{\hyp}$ on $\pi_1(\Sigma)$, inite first moment for $d_{\hyp}$ implies finite first moment for $D_{\hyp}$. 
By \Cref{t.Mah-Tio}, the drifts $a_1$ and $a_2$ are both positive.

\smallskip
Using the measure $\nu \times \hat{\nu}$, we can sample ordered pairs of points at infinity. 
With probability one, these points are distinct and hence determine a bi-infinite geodesic in $\mathbb{H}^2$. 
We parameterise the geodesic by $\gamma: (-\infty, \infty) \to \mathbb{H}^2$ such that $\gamma_0$ is the point of $\gamma$ closest to the base-point and $\gamma_t$ converges to the point at infinity sampled by $\nu$. 
As a consequence of the positivity of the drifts, it follows directly that
\begin{theorem}\label{t.random} 
Let $\mu$ be a non-elementary probability distribution on $\pi_1(\Sigma)$ with finite first moment and let $\hat{\mu}$ be the reflected distribution.
Let $\nu$ and $\hat{\nu}$ be the stationary measures on $S^1$ for $\mu$ and $\hat{\mu}$-random walks.
Then there exists a constant $K \geqslant 1$ such that for $\nu \times \hat{\nu}$-almost every $\gamma$, there is a $T_\gamma$ such that for all $T > T_\gamma$, 
\[
D_{\hyp} (\gamma_0, \gamma_T) \geqslant \frac{T}{K}.
\]
\end{theorem} 
Since the treatment in Maher--Tiozzo is quite general, we will provide a direct sketch for the positivity of the drifts later in the paper (see Section~\ref{s.mainproof1}).

\smallskip
The hyperbolic ray from the base-point $x$ that converges to the same point at infinity as the geodesic $\gamma$, is strongly asymptotic to $\gamma$. 
Thus, we deduce:
\begin{theorem}
Let $\mu$ be a non-elementary probability distribution on $\pi_1(\Sigma)$ with finite first moment and let $\nu$ be the stationary measure on $S^1$ for the $\mu$-random walk.
Then there exists $K \geqslant 1$ such that for $\nu$-almost every $\lambda \in S^1$, there exists $T_\lambda$ such that for any $T > T_\lambda$ the point $\gamma_T$ along the hyperbolic geodesic ray from $x$ that converges to $\lambda$, we have
\[
D_{\hyp} (\gamma_0, \gamma_T) \geqslant \frac{T}{K}.
\]
\end{theorem}
In other words, a typical ray in the fibre $\mathbb{H}^2$ makes linear progress in $\mathbb{H}^3$. 

%%%%%%%%%%%%%%%%%%%%%%%%%%%%%%%%%%%
\section{Flat and Solv geometry} 

In order to prove \Cref{t.linear-progress} and \Cref{t.random}, we analyse the geometry by starting with the flat and singular solv geometry.
A pseudo-Anosov map on $\Sigma$ acting on Teichm\"uller space has an invariant axis that it translates along. 
We may then consider a holomorphic quadratic differential $q$ along the axis and use contour integration of a square root of it to equip $\Sigma$ with a singular flat metric.
The singular flat metric lifted to the universal cover $\mathbb{H}^2$ of $\Sigma$ will be denoted by $d_{\fl}$. 
The singular flat metrics on all other fibres can be derived by the action of the corresponding affine maps along the Teichm\"uller axis. 
Put together, these metrics equip the universal cover $\mathbb{H}^3$ of $M$ with a singular solv metric which we denote by $d_{\sol}$. 

\subsection{Optimal shadows} 
We will define below the \emph{optimal shadow} associated to a flat geodesic based at a point on it.
The orientation on $\mathbb{H}^2$ induces a cyclic order on the unit tangent circle at any point in $\mathbb{H}^2$. 
We will make use of this cyclic order in the description.

\smallskip
Let $\beta: [0, T] \to \mathbb{H}^2$ be a parameterised flat geodesic and let $0 < t \leqslant T$. 
As a preliminary, we define the \emph{lower} and \emph{upper} unit tangent vectors to $\beta$ at $\beta_t$. 
Let $\epsilon > 0$ be small enough so that the segment $[\beta_{t- \epsilon}, \beta_{t+\epsilon}]$ contains no singularities except possibly at $\beta_t$. 
We then define the lower unit tangent vector $v^-(\beta_t)$ to be the unit tangent vector to $[\beta_{t-\epsilon}, \beta_t]$ at the point $\beta_t$. 
Similarly, we define the upper unit tangent vector $v^+(\beta_t)$ to be the unit tangent vector to $[\beta_t, \beta_{t+\epsilon}]$ at the point $\beta_t$. 
Note that if $\beta_t$ is a regular point then $v^-(\beta_t) = v^+(\beta_t)$. 

\smallskip
We will first show the existence of an \emph{lower perpendicular}. 
Let $\beta: [0,T] \to \mathbb{H}^2$ be a parameterised flat geodesic and let $0< t \leqslant T$. 
We say that a flat geodesic segment $\alpha: (-s, s) \to \mathbb{H}^2$ is the \emph{lower perpendicular} to $\beta$ at $\beta_t$ if
\begin{itemize}
\item $\alpha_0 = \beta_t$; and
\item the unit tangent vectors $v^{\pm}(\alpha_0)$ make an angle of $\pi/2$ with the lower tangent $v^-(\beta_t)$. 
\end{itemize} 

Let $\beta'$ be a bi-infinite perpendicular to $\beta$ at $\beta_t$. 
We denote the component of $\mathbb{H}^2 - \beta'$ that contains $\beta_0$ by $C^-(\beta')$. 

\begin{lemma}
Let $\beta: [0,T] \to \mathbb{H}^2$ be a parameterised flat geodesic and let $0 < t \leqslant T$. There exists a bi-infinite flat geodesic $\beta_\perp^t$ such that 
\begin{itemize}
\item $\beta_\perp^t$ is the lower perpendicular to $\beta$ at $\beta_t$; and
\item for any bi-infinite geodesic $\beta'$ perpendicular to $\beta$ at $\beta_t$
\[
C^-(\beta_\perp^t) \subseteq C^- (\beta').
\]
\end{itemize}
\end{lemma} 

\begin{proof}
Breaking symmetry, suppose that $\beta_t$ is a regular point. 
Then $\beta$ has exactly two perpendicular directions $v_L, v_R$ at $\beta_t$. 
Using the cyclic order on the unit tangent circle at $\beta_t$, we arrange matters so that $v_L, v^-(\beta_t), v_R$ in that order are counter-clockwise. 

\smallskip
If the flat ray with initial direction $v_L$ is infinite then we set $\beta_\perp^L$ to be this ray. 
So suppose that the ray with initial vector $v_L$ runs in to a singularity $p$ in finite time.
Let $v^-$ be the lower tangent vector to the ray at the point $p$. 
In the induced cyclic order on the unit tangent circle at $p$, we may move clockwise from $-v^-$ till we get a vector $v^+$ that is at an angle $\pi$ from $-v^-$. 
We then extend the initial ray by the flat ray with initial vector $v^+$. 
Continuing iteratively in this manner, we obtain an infinite flat ray which we set to be $\beta_\perp^L$.

\smallskip
We may then carry out an analogous construction with $v_R$ to obtain an infinite flat ray $\beta^\perp_R$ with initial vector $v_R$. In the analogous construction, should we encounter a singularity, we move counter-clockwise from $-v^-$ by an angle of $\pi$ to continue.

\smallskip
We then set $\beta_\perp^t$ to be the union $\beta_\perp^L \cup \beta_\perp^R$. As the angle between $\beta_\perp^L$ and $\beta_\perp^R$ at the point $\beta_t$ is exactly $\pi$, the union is a bi-infinite flat geodesic.

\smallskip
Suppose instead that $\beta_t$ is a singularity. 
Using the induced cyclic order on the unit tangent circle at $\beta_t$, we move clockwise from $-v^-(\beta_t)$ till we are at the vector $v_L$ that is at angle $\pi/2$ from $-v^-(\beta_t)$. Similarly, we move counter-clockwise from $-v^-(\beta_t)$ till we are at the vector $v_R$ that is at angle $\pi/2$ from $-v^-(\beta_t)$.
We now proceed to construct rays $\beta_\perp^L$ (and $\beta_\perp^R$) with initial vectors $v_L$ (respectively $v_R$) exactly as above.
We then set $\beta_\perp^t$ to be again the union $\beta_\perp^L \cup \beta_\perp^R$. 
The angle between $\beta_\perp^L$ and $\beta_\perp^R$ is $\pi$ and hence $\beta_\perp^t$ is a bi-infinite geodesic.

\smallskip
Finally, suppose $\beta'$ is a bi-infinite geodesic perpendicular to $\beta$ at $\beta_t$. If $\beta'$ is distinct from $\beta$, then it diverges from $\beta_\perp^t$ at some singularity. Breaking symmetry, suppose that there is a singularity $p$ along $\beta_\perp^L$ at which $\beta'$ diverges from $\beta_\perp^L$. Then, the angle that $\beta'$ makes at $p$ exceeds $\pi$. 

\smallskip
Let $v^+(\beta') \neq v^+(\beta_\perp^L)$ be the tangent vectors at $p$ to $\beta'$ and $\beta_\perp^L$.
Suppose that the rays with these initial vectors intersect. 
Then the rays give two geodesic segments that bound a bigon. 
As the metric is flat with the negative curvature concentrated at singularities, the presence of this bigon contradicts Gauss-Bonnet. 
Hence, the rays do not intersect. 
Let $s_L$ then be the region bound by this pair of rays and not containing $\beta_0$.

\smallskip
A similar argument applies if there is divergence between $\beta'$ and $\beta_\perp^R$ and yields a region $s_R$ that does not contain $\beta_0$. 

\smallskip
We deduce that $C^-(\beta') = C^-(\beta_\perp^t) \cup s_L \cup s_R$ thus finishing the proof.

\end{proof} 

We also present a slightly less direct construction for $\beta_\perp^t$.
The flat metric on $\Sigma$ has finitely many singularities and so their lifts to $\mathbb{H}^2$ yield a 
countable discrete set in $\mathbb{H}^2$. 
Let $t$ in $(0,T]$ be a time such that there exist $\epsilon > 0$ depending on $t$ such that the segment $[\beta_{t - \epsilon}, \beta_t]$ contains no singularity. 
We may then consider the foliation $\lambda_\perp$ be the foliation that is perpendicular to the segment $[\beta_{t - \epsilon}, \beta_t]$. 
Among such times, we say $t$ is \emph{simple} if the leaf of $\lambda_\perp$ containing $\beta_t$ is bi-infinite and does not contain any singularities. 
By the observation regarding lifts of singularities, the set of simple times is dense (in fact, full measure) in $[0,T]$.
Let $s < t$ be simple times. 
A Gauss-Bonnet argument similar to the one in the proof above shows that $C^- (\beta_\perp^s) \subset C^- (\beta_\perp^t)$. 
Suppose now that $t > 0$ is a time that is not simple. 
We then consider a sequence of simple times $t_n < t$ such that $t_n$ converges to $t$ and define $\beta_\perp^t$ as the limit of the bi-infinite perpendiculars at $\beta_{t_n}$. 
We leave it as a simple exercise to check that this reproduces our definition in the above lemma.

\smallskip
The bi-infinite flat geodesic $\beta_\perp^t$ divides $\mathbb{H}^2$ in to two components. We call the component of $\mathbb{H}^2 - \beta_\perp^t$ that does not contain $\beta_{t- \epsilon}$ the \emph{optimal shadow} of $\beta$ at the point $\beta_t$. We denote the optimal shadow by $\shadow(\beta_t)$. 

\smallskip
A number of easy consequences follow from Gauss-Bonnet. 

\begin{lemma}\label{l.convex-shadow} 
Let $\beta: [0,T] \to \mathbb{H}^2$ be a parameterised flat geodesic and let $0 < t \leqslant T$. 
The optimal shadow $\shadow(\beta_t)$ is convex in the flat metric on $\mathbb{H}^2$. 
\end{lemma}

\begin{proof}
Let $x, x'$ be distinct points in $\shadow(\beta_t)$ and suppose that the flat geodesic segment $[x,x']$ intersects $\beta_\perp^t$. 
Then the number of intersection points is at least two.

\smallskip
We may parameterise $[x,x']$ and consider consecutive points of intersection. 
Between these points, $[x, x']$ and $\beta_\perp^t$ bound a nontrivial bigon.
The presence of such a bigon contradicts the Gauss-Bonnet theorem.

\end{proof} 

As a consequence of \Cref{l.convex-shadow}, we immediately deduce the following lemma.
\begin{lemma}\label{l.nesting} 
Let $\beta: [0,T] \to \mathbb{H}^2$ be a parameterised flat geodesic and let $0 < s < t \leqslant T$. Then $\shadow(\beta_t)$ is nested strictly inside $\shadow(\beta_s)$, that is, $\shadow(\beta_t) \subset \shadow(\beta_s)$ and $\beta_\perp^t$ is contained in the interior of $\shadow(\beta_s)$. 
\end{lemma}

\begin{proof}
Suppose that $\beta_\perp^s$ and $\beta_\perp^t$ intersect. 
Breaking symmetry, we may assume that there is an intersection point to the left of $\beta$. 
Let $p$ be the first point of intersection to the left. 

\smallskip
Let $[\beta_s, p]$ (respectively, $[\beta_t, p]$) be the finite segment of $\beta_\perp^s$ (respectively, $\beta_\perp^t$) with endpoints $\beta_s$ (respectively, $\beta_t$) and $p$. 

\smallskip
Since $\beta_\perp^s$ and $\beta_\perp^t$ are both perpendicular to $\beta$, the sum of the angles of the triangle with sides $[\beta_s, p], [\beta_t, p]$ and $[\beta_s, \beta_t]$ exceeds $\pi$. Since the metric is flat with negative curvature at the singularities, the presence of such a triangle violates the Gauss--Bonnet theorem. 

\smallskip
Hence, we may conclude that $\beta_\perp^s$ and $\beta_\perp^t$ do not intersect and the lemma follows.

\end{proof}

In the next two lemmas we justify the sense in which $\shadow(\beta_t)$ is actually a shadow.

\begin{lemma}\label{l.closest}
For any parameterised geodesic $\beta: [0, T] \to \mathbb{H}^2$ and any $0 < t \leqslant T$, the point on $\beta_\perp^t$ that is closest to $\beta_s$ for any $s \in [0,T]$ is $\beta_t$. 
\end{lemma} 

\begin{proof}
Suppose a point $x \in \beta_\perp^t$ distinct from $\beta_t$ is closest to $\beta_s$. 
Consider the triangle with sides $[\beta_s, \beta_t]$, $[\beta_t, x]$ and $[x, \beta_s]$, where
$[\beta_t, x]$ is a sub-segment of $\beta_\perp^t$. % connecting the two points.
Since $x$ is closest to $\beta_s$, the angle inside the triangle at $x$ is $\pi/2$, as otherwise
we could shorten the path from $\beta_s$ to $x$. 
Hence, the triangle with our chosen sides has two of its angles $\pi/2$, which violates the Gauss--Bonnet theorem.

\end{proof} 

\begin{lemma}\label{l.fellow-travel}
There exists $r_0> 0$ such that for any $t $ satisfying $0 < r_0 < t$, any parameterised geodesic $\beta: [0, t] \to \mathbb{H}^2$ and any $x \in \shadow(\beta_t)$, the flat geodesic segment $[\beta_0, x]$ intersects the ball $B(\beta_t, r_0)$. 
\end{lemma} 

\begin{proof}
By \Cref{l.convex-shadow}, shadows $\shadow(\beta_t)$ are convex and by \Cref{l.closest} the segment $[\beta_0, \beta_t]$ gives the closest point projection. The existence of $r_0>0$ then follows from the %$\text{CAT}(0)$ property 
hyperbolicity of the singular flat metric.

\smallskip
Alternatively, given $x \in \shadow(\beta_t)$, we can give a more detailed description of the flat geodesic $[x, \beta_0]$. 
Note that it suffices to assume $x \in \beta_\perp^t$. 

\smallskip
We will recall some facts from flat geometry to give this description.
For any point $y$ on the fibre $S$, consider the set $\text{Sad}(y)$ of flat geodesic arcs that 
\begin{itemize}
\item join $y$ to a singularity; and 
\item have no singularity in their interior.
\end{itemize}

It is a standard fact in the theory of half-translation surfaces that the slopes of such arcs equidistribute in the set of directions. See \cite{Mas1} and \cite{Mas2}. 
It follows that we can find $r > 0$ such that all gaps in the slopes of all arcs in $\text{Sad}(y)$ with length at most $r$, are less than $\pi/2$.

\smallskip
Now we consider $\shadow(\beta_t)$ and breaking symmetry, consider $\beta_\perp^L$. 
Suppose that no flat geodesic segment $[x, \beta_0]$, where $x \in \beta_\perp^L$, passes through a singularity in its interior. This then means that the sector of angle $\pi/2$ based at $\beta_t$, with sides $[\beta_t, \beta_0]$ and $\beta_\perp^L$, contains no arc in $\text{Sad}(\beta_t)$ with length at most $r$, a contradiction.

\smallskip
Moving along $\beta_\perp^L$ away from $\beta_t$, let $x$ be the first point for which $[x, \beta_0]$ passes through a singularity $p$.
It follows that for later points $x'$ along $\beta_\perp^L$, the geodesic segments $[x', \beta_0]$ must pass through $p$. 

\smallskip
The same holds for points in $\beta_\perp^R$ and concludes our proof.

\end{proof} 

We record the following consequence. 

\begin{lemma}\label{l.containment}
Let $r \geqslant r_0$, where $r_0$ is the constant in \Cref{l.fellow-travel}. Then there exists $\ell_0> 0$ that depends on $r$ such that for any $\ell \geqslant \ell_0$ and any flat geodesic segment $[\beta_0, \beta_{3\ell}]$ of flat length $3\ell$ and any flat geodesic segment $\beta'$ that fellow-travels $\beta$ so that after parameterising $\beta': (-\epsilon, t] \to \mathbb{H}^2$ to arrange $\beta'_0 \in B(\beta_0, r)$ and $\beta'_t \in B(\beta_{3\ell}, r)$ for some $t$ satisfying $3 \ell - 2r < t < 3 \ell + 2r$, we have that 
\[
\mathbb{H}^2 - \shadow(\beta'_0) \subseteq \mathbb{H}^2 - \shadow(\beta_\ell) 
\]
and
\[
\shadow(\beta'_t) \subseteq \shadow(\beta_{2\ell}).
\]
\end{lemma} 

\begin{proof} 
We give a proof of the first inclusion; the second inclusion follows using similar arguments. 

\smallskip
Given $r$, there exists $r' \geqslant r$ that depends only on $r$ such that $r'$ is the fellow travelling constant for the segments $[\beta'_0, \beta'_t]$ and $[\beta_0, \beta_{3\ell}]$. 
If $\ell_0 > 2r'$ then $\beta'^0_\perp$ does not intersect the ball $B(\beta_0, r')$. Otherwise, for any point of $\beta'$ that lies in $B(\beta_\ell, r')$ the point $\beta'_0$ is not the closest point on $\beta'^0_\perp$, contradicting \Cref{l.closest}.
Suppose now that $\beta'^0_\perp$ and $\beta_\perp^\ell$ intersect in the point $p$.
Since $\beta'^0_\perp$ does not intersect $B(\beta_0, r')$, the geodesic segment $[\beta'_0, p] \subset \beta'^0_\perp$ does not intersect $B(\beta_\ell, r_0)$ which contradicts \Cref{l.fellow-travel}. 
Hence, the geodesics $\beta'^0_\perp$ and $\beta_\perp^\ell$ do not intersect when $\ell_0 > 2 r'$, from which we deduce $\mathbb{H}^2 - \shadow(\beta'_0) \subseteq \mathbb{H}^2 - \shadow(\beta_\ell)$.

\end{proof} 

\subsection{Ladders}

The universal cover $\mathbb{H}^3$ can be equipped with the $\mathbb{Z}$-equivariant pseudo-Anosov flow $\{\psi_r \, ; \, r \in \mathbb{R} \}$ such that the time 1-map is the lift of the pseudo-Anosov monodromy $f$ of the fibered 3-manifold $M$.
In fact, various lifts to the universal covers of the fibre inclusion $\Sigma \to M$ are precisely given by $\psi_r$ applied to our chosen lift $\mathbb{H}^2 \to \mathbb{H}^3$. 
We will call these lifts the \emph{$r$-fibres} in $\mathbb{H}^3$. 
As a notational choice, we will denote the inclusion of $\mathbb{H}^2$ in $\mathbb{H}^3$ given by the $r$-fibre by $\psi_r(\mathbb{H}^2)$. 

\begin{definition}
Let $\beta$ be a flat geodesic segment in $\mathbb{H}^2$. 
We define the \emph{ladder} given by $\beta$ to be the set
\[
\lad(\beta) = \bigcup_{r \in \mathbb{R}} \psi_r(\beta)
\]
\end{definition} 

Comparing our definition to the ladders introduced by Mitra in \cite{Mit}, we note two differences: 
\begin{enumerate}
\item In \cite{Mit}, the ladders are constructed for the group extension $1 \to \pi_1(\Sigma) \to \pi_1(M) \to \mathbb{Z} \to 1$, and thus the ladder consists of one segment in each $\mathbb{Z}$-coset of $\pi_1(\Sigma)$. We operate directly in $\mathbb{H}^2$ and $\mathbb{H}^3$ and our ladders contain the segment $\psi_r(\beta)$ in each fibre $\psi_r(\mathbb{H}^2)$. 
As $\pi_1(\Sigma)$ and $\pi_1(M)$ act co-compactly on $\mathbb{H}^2$ and $\mathbb{H}^3$, the two points of view are equivariantly quasi-isometric by the Svarc-Milnor lemma. 
\item Secondly, in \cite{Mit}, when the segment $\beta$ in $\pi_1(\Sigma)$ is moved to a different fibre, it is pulled tight in the metric on the coset. 
In our setup, the maps $\psi_r$ are affine maps. 
As a result, the segment $\psi_r(\beta)$ is already geodesic in the corresponding flat metric on $\psi_r( \mathbb{H}^2)$. 
\end{enumerate} 

\smallskip
Thus, \cite[Lemma 4.1]{Mit} applies, and in our notation this translates
\begin{lemma}\label{lem:Mahan}
For any geodesic $\beta$ in the flat metric, the ladder $\lad(\beta)$ is quasi-convex in the singular solv metric
on $\mathbb{H}^3$ with quasi-convexity constants independent of $\beta$. 
\end{lemma}
As the singular flat metric and the singular solv metric are quasi-isometric to the hyperbolic metrics on $\mathbb{H}^2$ and $\mathbb{H}^3$, quasi-convexity of ladders is also true for the hyperbolic metrics.
We will continue the geometric discussion for the singular flat and singular solv metrics.

\smallskip
We now consider the ladders defined by shadows, namely the sets
\[
L(\beta, t) = \coloneqq \bigcup_{x \in \shadow(\beta_t)} \bigcup_{r \in \mathbb{R}} \psi_r(x).
%\lad(\shadow(\beta_t) ) \coloneqq \bigcup_{x \in \shadow(\beta_t)} \bigcup_{r \in \mathbb{R}} \psi_r(x).
\]

By \Cref{l.nesting} and item (2) above the lemma below immediately follows.
\begin{lemma}\label{l.containment} 
Let $\beta: [0,T] \to \mathbb{H}^2$ be a parameterised flat geodesic and let $0 < s < t \leqslant T$. Then 
\[
L(\beta, t) \subsetneq L(\beta, s).
%\lad(\shadow(\beta_t)) \subsetneq \lad(\shadow(\beta_s)).
\]
\end{lemma}

\subsection{Undistorted segments} 

To quantify (un)distortion of segments, we use the following definition. 
\begin{definition}
  For a constant $K \geqslant 1$, we say that a flat geodesic segment $\beta: [0, T] \to \mathbb{H}^2$ is $K$-undistorted if $d_{\sol}(\beta_0 , \beta_T) \geqslant T/ K$. 
\end{definition}
The requirement that a segment be undistorted translates to the following geometric criterion.

%\smallskip
Let $\lambda_\pm$ be the stable and unstable foliations for the pseudo-Anosov monodromy $f$. 
If a flat geodesic segment makes a definite angle with both foliations $\lambda_\pm$ then it cannot be shortened beyond a factor that depends only on the bound on the angle. 
As a result, it is undistorted. 

\smallskip To be more precise, let $\beta = [\beta_0, \beta_t]$ be a flat geodesic segment. 
We may write $\beta$ as a concatenation $\beta = [\beta_0, \beta_{s_1}] \cup [\beta_{s_1}, \beta_{s_2}] \cup \dotsc [\beta_{s_{k-1}}, \beta_{s_k}]$, where $s_k = t$ and where $\beta_{s_j}$ is a singularity for all $1 \leqslant j \leqslant k-1$. 
We then let
\[
\slope(\beta) = \{ m_j \, : \, m_j \text{ is the slope of } [\beta_{s_{j-1}}, \beta_{s_j} ] \text{ for } 1 \leqslant j \leqslant k \} 
\]

One could now precisely quantify the maximal $K$ so that $\beta$ is $K$--undistorted in terms of $\slope(\beta)$, but since we do not make a particular use of it, we refrain from doing so.

%The distortion constant can be precisely quantified in terms of the slope of the segment with respect to $\lambda_\pm$.
%S

\smallskip
For our purposes, the following soft definition suffices. 

\begin{definition}\label{d.good}
For $\ell > 0$, we say that a flat geodesic segment $\beta$ of flat length is \emph{$\ell$-good} if 
\begin{itemize}
\item the length of $\beta$ is at least $3\ell$; and
\item all slopes $m$ in $\slope(\beta)$ satisfy $1/2 < | m | < 2$. 
\end{itemize} 
\end{definition}
The choice of the slope bound is arbitrary and it will only affect the constant of undistorted and other coarse constants to follow.

\smallskip
We first list a consequence of the slope bounds.
\begin{lemma}\label{l.quasi-geo}
Let $\beta: [0, t] \to \mathbb{R}$ be a flat geodesic segment with slope $m$ satisfying $1/2 < | m |< 2$.  
Then $\beta$ is a quasi-geodesic in $\lad(\beta)$ and there is a constant $c \coloneqq c(m)$ such that $d_{\sol}(\lad(\beta_0), \lad(\beta_t)) \geqslant c \cdot t$.
\end{lemma}

\begin{proof}
The ladder $\lad(\beta)$ intersects the $r$-fibre in $\psi_r(\beta)$.
Since $\psi_r$ is pseudo-Anosov flow, the slope bounds imply that there is a constant $c> 0$ that depends only on the bounds, such that for any $r \in \mathbb{R}$ the flat length of $\psi_r(\beta)$ in the fibre $\psi_r(\mathbb{H}^2)$ is at least $c \cdot l_{\fl}(\beta)$. 
We deduce that any path in $\lad(\beta)$ that connects $\lad(\beta_0)$ with $\lad(\beta_t)$ must also have length at least $c \cdot l_{\fl} (\beta)$. The lemma follows. 

\end{proof} 

From \Cref{l.quasi-geo}, we deduce the lemma below, the key reason why we need the notion of $\ell$--good segments.
\begin{lemma}\label{lem:good-segments-good}
	There is a constant $B > 0$ such that every $\ell$--good segment is a $B$--quasi-geodesic in the singular solv metric on $\mathbb{H}^3$.
\end{lemma}
\begin{proof}
	by Lemma~\ref{lem:Mahan}, ladders are quasi-convex in the singular solv metric on $\mathbb{H}^3$, hence undistorted. Thus, a singular solv geodesic from $\lad(\beta_0)$ to $\lad(\beta_t)$ is contained in a bounded neighbourhood of $\lad(\beta)$. So it suffices to show that if $\beta$ is $L$--good,
	then it is a quasi-geodesic in $\lad(\beta)$.
	To this end, we may write $\beta$ as a concatenation
	\[ 
	\beta = [\beta_0, \beta_{s_1}] \cup [\beta_{s_1}, \beta_{s_2}] \cup \dotsc [\beta_{s_{k-1}}, \beta_{s_k}].
	\]
	Then
	\[ 
	\lad(\beta) = \lad([\beta_0, \beta_{s_1}]) \cup \lad([\beta_{s_1}, \beta_{s_2}]) \cup \dotsc \lad([\beta_{s_{k-1}}, \beta_{s_k}]). 
	\]
         and by \Cref{l.quasi-geo} each segment $[\beta_{s_{j-1}}, \beta_{s_j}]$ is quasi-geodesic in the corresponding sub-ladder $\lad([\beta_{s_{j-1}}, \beta_{s_j}])$. 
         Also, by \Cref{l.quasi-geo}, each ladder $\lad(\beta_{s_j})$ is well-separated from its predecessor $\lad(\beta_{s_{j-1}})$ and its successor $\lad(\beta_{s_{j+1}})$ by distances that are linear in the lengths of the segments $[\beta_{s_{j-1}}, \beta_{s_j}]$ and $[\beta_{s_j}, \beta_{s_{j+1}}]$, and hence the concatenation is a quasi-geodesic. 
        
         \end{proof}

         \begin{lemma}\label{lem:existence-good-segments}
           For any $\ell > 0$ there is an $\ell$--good segment.
         \end{lemma}
         \begin{proof}           
           The discreteness of saddle connection periods and the quadratic growth asymptotic in every sector of slopes for the number of saddle connections counted by length implies the existence of a good segment. 
           See \cite{Mas1}, \cite{Mas2} for these facts.
           In fact, this shows that a single saddle connection can be chosen as a good segment instead of a concatenation. We also remark that the existence could be shown with more elementary means, but we refrain from doing so for brevity. 
         \end{proof}
         
         We now derive nesting along ladders of shadows along $\ell$-good segments (compare Figure~\ref{fig:flat-shadows}).
         
 \begin{figure}[h!]
  \centering
  \includegraphics[width=\textwidth]{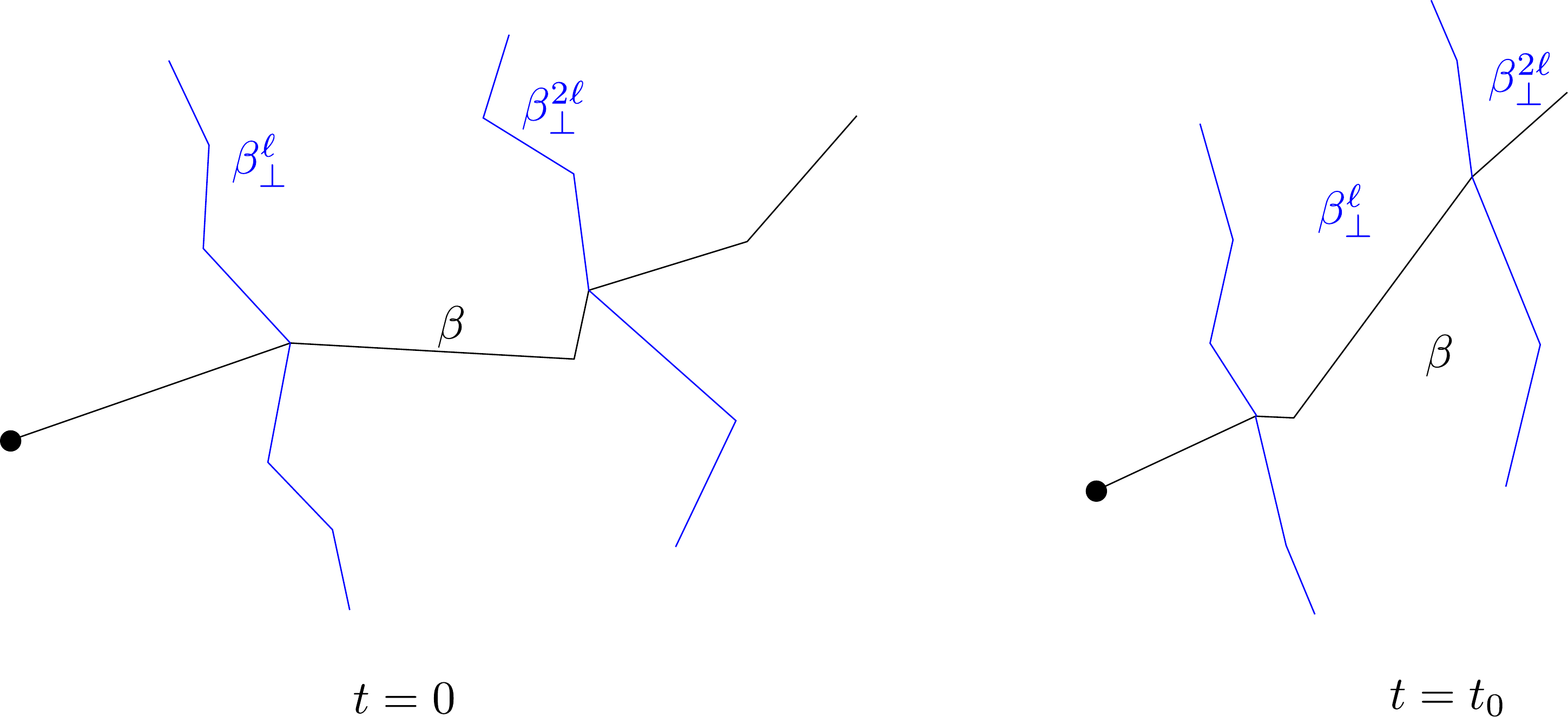}
  \caption{A good segment forces separation of shadows over all fibres (Lemma~\ref{l.nesting-ladders}). The left part of the figure shows the configuration in the base fibre. The middle segment of $\beta$ is assumed to be $\ell$--good. Moving to a different fibre, shown on the right, changes the geometry, but goodness of the middle segment ensures that the distance between (the images of) $\beta^l_\perp, \beta^{2l}_\perp$ is still large, leading to definite separation of the corresponding shadows.}
  \label{fig:flat-shadows}
\end{figure}

\begin{lemma}\label{l.nesting-ladders}
There exists a constant $D \geqslant 1$ such that for $\ell$ sufficiently large, any $\ell$-good segment $\beta$ 
\[
\frac{1}{D} \ell \leqslant d_{\sol}(\lad(\beta_\perp^\ell), \lad(\beta_\perp^{2\ell}) ) \leqslant D \ell.
\]
\end{lemma}

In other words, the flat geodesic segment $\beta$ is coarsely the line of nearest approach for the ladders, that is, it gives coarsely the shortest distance between the ladders.
%and also from the base-point to the outer ladder.

\begin{proof}
By \Cref{lem:good-segments-good}, $\beta$ is a quasi-geodesic in the singular solv metric. The upper bounds follow immediately from this.

\smallskip
By construction, at every singularity contained in $\beta_\perp^\ell$, the angle subtended on one side is exactly $\pi$.
Similarly for $\beta_\perp^{2 \ell}$. 
Because of the constraints on $\slope(\beta)$, it follows that all slopes $m$ in $\slope(\beta_\perp^\ell)$ and in $\slope(\beta_\perp^{2\ell}) $ also satisfy $1/2 < | m | < 2$. By \Cref{lem:good-segments-good}, $\beta, \beta_\perp^\ell$ and $\beta_\perp^{2 \ell}$ are quasi-geodesics in $\mathbb{H}^3$ for the singular solv metric. 

\smallskip
Let $x \in \beta^\ell_\perp$ and $x' \in \beta^{2\ell}_\perp$. 
By arguing as in \Cref{l.fellow-travel}, the flat geodesic segment $[x, x']$ fellow travels the concatenation $[x, \beta_\ell] \ast [\beta_\ell, \beta_{2\ell}] \ast [\beta_{2\ell}, x']$, where $[x, \beta_\ell] \subset \beta^\ell_\perp$ and $[\beta_{2\ell}, x'] \subset \beta^{2\ell}_\perp$. 
This implies that if $\ell$ is sufficiently large all slopes in $[x, x']$ are bounded away from the horizontal and vertical. 
Thus, $[x, x']$ is also a quasi-geodesic in the singular solv metric and the result follows from this and the fellow-travelling.

\end{proof}

\section{Hyperbolic geometry}

We now pass to the hyperbolic metrics on the fibre and the 3-manifold.
Recall that we denote the hyperbolic metric on $\mathbb{H}^2$ by $d_{\hyp}$ and the hyperbolic metric on $\mathbb{H}^3$ by $D_{\hyp}$. 

%% Put the appropriate citation below
\smallskip
The metrics $(\mathbb{H}^2, d_{\fl}) $ and $( \mathbb{H}^2, d_{\hyp}) $ are quasi-isometric; so are the metrics $(\mathbb{H}^3, d_{\sol})$ and $(\mathbb{H}^3, D_{\hyp})$. 
Let $(K_1, A_1)$ and $(K_2, A_2)$ be the quasi-isometry constants in each case and set $K= \max \{ K_1, K_2\}, A = \max \{ A_1, A_2\} $.
%We remark that (as the proof of \Cref{l.fellow-travel} clarifies) the quasi-isometry constants depend on the pseudo-Anosov monodromy $f$. For instance, if the flat structure has a very wide annulus then the singular flat metric has a metric flat strip of that width. 

\smallskip
As a consequence of these quasi-isometries, we can recast \Cref{l.convex-shadow} and the quasi-convexity of ladders in $(\mathbb{H}^3, d_{\sol})$ in the previous section to conclude that for any parameterised flat geodesic $\beta: [0,T] \to (\mathbb{H}^2, d_{\fl})$ and any $0 < t \leqslant T$
\begin{itemize}
\item $\shadow(\beta_t)$ is quasi-convex in $(\mathbb{H}^2, d_{\hyp})$; and
\item $L (\beta, t) = \lad(\shadow(\beta_t))$ is quasi-convex in $(\mathbb{H}^3, D_{\hyp})$.
\end{itemize}
Let $C_1$ be the quasi-convexity constant in the first instance and $C_2$ the quasi-convexity constant in the second instance.
Set $C = \max \{ C_1, C_2 \}$.

\subsection{Non-backtracking} 

Let $\gamma: (\infty, \infty) \to \mathbb{H}^2$ be a hyperbolic geodesic parameterised by unit speed with $x^{-\infty}$ and $x^\infty$ in $S^1$ its points at infinity, where $\gamma_t \to x^{\pm \infty}$ as $t \to \pm \infty$.  
Let $\beta$ be a bi-infinite flat geodesic that also converges to $x^{\pm \infty}$. 
Note that $\beta$ might not be unique
%two core geodesics of a flat cylinder on $\Sigma$ lift to a pair of distinct bi-infinite geodesics that have the same points at infinity
but any such geodesic fellow-travels $\gamma$ in the hyperbolic metric. 
By resetting the constant $C$, we may assume that the fellow-travelling constant in both hyperbolic and flat metrics can also be chosen to be $C$. 

\medskip
By the same proof as \Cref{l.closest}, there is a unique point $p$ in $\beta$ that is closest to the point $\gamma_0$ in the flat metric. 
We parameterise $\beta$ with unit speed such that $\beta_s \to x^{\pm \infty}$ as $s \to \pm \infty$ and $\beta_0 = p$.

\medskip
For any time $t$ along $\gamma$ such that $t/K - A \geqslant 0$, the flat distance between $\gamma_0$ and $\gamma_t$ is at least $t/K -A$. 
Thus, the flat distance between $\beta_0$ and $\gamma_t$ is at least $t/K-A -C$. 
Thus, for any $t$ that satisfies $(t/K) - A \geqslant 2C$, the point $\gamma_t$ is contained in $\shadow(\beta_{s(t)})$ where
\begin{equation}\label{e.assigned-time}
s(t) = \frac{t}{K}  - A - 2C. 
\end{equation}

Recalling our notation $L(\beta, u) = \lad(\shadow(\beta_u))$, define the function $f_\gamma: \mathbb{R}_{\geqslant 0 } \to \mathbb{R}_{\geqslant 0}$ by 
\[
f_\gamma(t) = D_{\hyp} (\gamma_0, L(\beta, s(t)))
\]
and note that since $\gamma_t$ is contained in $\shadow(\beta_{s(t)})$, we have $f_\gamma(t) \leqslant D_{\hyp} (\gamma_0, \gamma_t)$. 

\smallskip
Since $s(t) < s(t')$ whenever $t< t'$ and since the ladders of nested shadows are nested, we have $f_\gamma(t) \leqslant f_\gamma(t')$, that is, $f_\gamma$ is a non-decreasing function of $t$. 
To prove \Cref{t.linear-progress} it then suffices to prove that $f_\gamma(t) $ grows linearly in $t$. 

\subsection{Progress certificate}\label{s.progress}
%Reseting the constant $C$, we may assume that $C$ is also the hyperbolic fellow-travelling constant in \Cref{l.flat-hyp}.
We fix $\rho > 0$ and set the constant $r$ in \Cref{l.containment} to be $r = K \rho + A+C$. 
With this value of $r$, we choose $\ell > 0$ to be large enough so that both \Cref{l.containment} and \Cref{l.nesting-ladders} hold.
By increasing $\ell$ further, we may assume that $\ell/ (K D) - A > 0$ and then set $R = \ell/(KD) - A$. 
We now choose an $\ell$-good segment $\beta$. 

\smallskip
We now define a subset in $T^1(\Sigma)$ that will certify progress in the hyperbolic metric.
Let $B(\beta_0, \rho)$ the ball with radius $\rho> 0$ in the hyperbolic metric centred at $\beta_0$. 
Let $V$ be the subset of $T^1 B(\beta_0, \rho)$ consisting of those unit tangent vectors $v$ such that the forward geodesic ray $\gamma_t = g_t v $ passes through the hyperbolic ball $B(\beta_{3 \ell}, \rho)$ centred at $\beta_{3\ell}$. 
Let $\Lambda$ be the image in $T^1(\Sigma)$ of $V$ under the covering projection.

\smallskip
Extending the hyperbolic geodesic $\gamma$ considered in the above paragraph to make it bi-infinite, let $\beta'$ be any flat bi-infinite geodesic that converges to the same points at infinity as $\gamma$.
We may parameterise $\beta'$ with unit flat speed so that 
\begin{itemize}
\item $\beta'_s$ and $\gamma_s$ converge to the same point at infinity as $s \to \infty$; and 
\item $\beta'_0$ is the point on $\beta'$ closest to $\gamma_0$ in the flat metric.
\end{itemize} 
By our choice of constants it follows that $d_{\fl} (\beta'_0, \beta_0) < r $ and $d_{\fl} (\beta'_t, \beta_{3 \ell}) < r$ for some $t$ satisfying $3\ell - 2r < t < 3\ell + 2r$. 
Then, by \Cref{l.containment}, $\mathbb{H}^2 - \shadow(\beta'_0) \subseteq \mathbb{H}^2 - \shadow(\beta_\ell)$ and $\shadow(\beta'_t) \subseteq \shadow(\beta_{2\ell})$. 
By the choice of constants and \Cref{l.nesting-ladders}, it follows that 
\[
D_{\hyp}(\mathbb{H}^3 - L(\beta', 0), L(\beta', t)) \geqslant  D_{\hyp}(\lad(\shadow(\beta_\perp^\ell), \lad(\shadow(\beta_\perp^{2\ell})) \geqslant R.
\]
 %By the choice of constants and \Cref{l.nesting-ladders}, it follows that $D_{\hyp} (\beta'_0, L(\beta', t))   > R$.

\smallskip
We now suppress the discussion on the good segment to summarise the conclusions as follows. 
\begin{remark}\label{r.progress}
Let $\gamma$ be a bi-infinite hyperbolic geodesic and let $\beta$ be a parameterised flat geodesic such that it converges to the same points at infinity as $\gamma$ forwards and backwards and $\beta_0$ is the closest point on $\xi$ in the flat metric to $\gamma_0$. 
Let $p(t)$ be the flat time such that $\beta_{p(t)}$ is the point of $\beta$ closest to $\gamma_t$. 
There exists constants $r, \ell > 0$ such that if after projecting to $T^1(\Sigma)$ the unit tangent vector $v(\gamma_t)$ is in $\Lambda$, then $D_{\hyp}(\mathbb{H}^3 - L(\beta, p(t)), L(\beta, p')) \geqslant R$ for some $p' \leqslant p(t)+ 3\ell+2r$.
\end{remark}

%\smallskip
%As a consequence, if $\eta$ is any hyperbolic geodesic so that its tangent vector $v(\eta_s)$ at time $s$ is in $\Lambda$, and $\xi$ is the corresponding flat geodesic, then we have that $D_{\hyp}(\xi_0, \xi_{s'}) \geqslant D_{\hyp}(\xi_0, \xi_s) + R$.

\smallskip
Let $\mu_{\lio}$ be the Liouville measure on $T^1(\Sigma)$. 
We may normalise the measure to be a probability measure.
Then, note that $\mu_{\lio}(\Lambda) > 0$. 
We set $m = \mu_{\lio}(\Lambda)$. 

\section{Linear progress in the fibre for a fibered hyperbolic 3-manifold}
\label{s.mainproof1}

We are now derive linear progress, namely \Cref{t.linear-progress}.

\begin{proof}[Proof of \Cref{t.linear-progress}]
Let $\gamma$ and $\beta$ be respectively hyperbolic and flat geodesics as in \Cref{r.progress}.

\smallskip
Recall \Cref{e.assigned-time} for $s(t)$. 
Given $t> K(A+2C)$, the set of times $\{ u \, : \, p(u) \leqslant s(t) \} $ is closed and bounded above.
Let $w$ be its maximum and note that $w \geqslant (1/K) s(t) - A -2C$. 

\smallskip
Let $n_\gamma(w)$ be the number of visits by $\gamma$ to $\Lambda$ (in the sense above) till time $w$. 
From \Cref{r.progress}, we conclude $D_{\hyp} (L(\beta, 0), L(\beta, p(w))) \geqslant (n_\gamma(w) - 1) R$.

\smallskip
Since $p(w) \leqslant s(t)$, \Cref{l.containment} implies $D_{\hyp} (L(\beta, 0), L(\beta, s(t))) \geqslant (n_\gamma(w) - 1) R$.
We deduce
\begin{equation}\label{e.progress} 
f_\gamma(t) = D_{\hyp} (\gamma_0, L(\beta, s(t)) \geqslant (n_\gamma(w) - 1) R - C.
\end{equation}

\smallskip
Let $\chi_\Lambda$ be the characteristic function of $\Lambda$. 
In each visit $\gamma$ spends time at most $2\rho$ in $\Lambda$.
Hence
\begin{equation}\label{e.discrete-continuous} 
n_\gamma(w) \geqslant \frac{1}{2\rho} \int\limits_0^w \chi_\Lambda ( v(\gamma_t)) \, dt,
\end{equation}

\smallskip
By the ergodic theorem, for $\mu_{\lio}$-almost every $v \in T^1 \Sigma$, any lift $\gamma$ in $\mathbb{H}^2$ of the hyperbolic ray determined by $v$, satisfies
\[
\lim_{w \to \infty} \, \frac{1}{w} \int\limits_0^w \,  \chi_\Lambda ( g_t v) \, dt = m
\]
In particular, there exists a time $w_v > 0$ depending only on $v$ such that 
\begin{equation}\label{e.time-bound}
\frac{1}{w} \int\limits_0^w \chi_\Lambda ( g_t v) \, dt \geqslant \frac{m}{2}
\end{equation} 
for all $w > w_v$. 
Given $v$, let $t_v$ be a time along $\gamma$ such that $s(t_v) \geqslant p(w_v)$. 

\smallskip
By combining \Cref{e.progress}, \Cref{e.discrete-continuous} and \Cref{e.time-bound}, we conclude that for $\mu_{\lio}$-almost every $v \in T^1 \Sigma$, along any lift $\gamma$ in $\mathbb{H}^2$ of the hyperbolic geodesic ray determined by $v$, we get
\[
f_\gamma(t) \geqslant \left( \frac{m w}{4\rho} -1 \right) R  
\]
for all $t> t_v$. Since $w \geqslant (1/K) s(t) - A - 2C = (1/K^2) T - \text{constants}$, we conclude the proof of \Cref{t.linear-progress}.

\end{proof}

\begin{proof}[Proof of \Cref{t.random}]

We record some observations from flat geometry. 
Let $\beta[x, p]$ be a flat geodesic ray from the base-point $x$ to a point $p \in S^1 = \partial \mathbb{H}$.
We may write $\beta$ as a (possibly infinite) concatenation $\beta_1 \cup \beta_2 \cup \cdots$ of saddle connections, or slightly more precisely, where $\beta_{j-1} \cap \beta_j$ is a singularity for all $j \geqslant 2$. 
The vertical and horizontal foliations of the pseudo-Anosov monodromy $f$ have no saddle connections. 
We infer that only the initial segment $\beta_1$ and in case of a finite concatenation the terminal segment $\beta_n$ can possibly be vertical/ horizontal.  
We then define the \emph{tilted length} of $\ell_{\text{tilt}} (\beta)$ to be 
\[
\ell_{\text{tilt}} (\beta) = \sum\limits_{m_j \in \slope(\beta) \, : \, 0 < | m_j | < \infty} \ell(\beta_j).
\]
By the discreteness of saddle connection periods, it follows that given $\ell > 0$, the set of $p \in S^1$ such that $\ell_{\text{tilt}} (\beta[x, p]) < 2 \ell$, is finite. 

\smallskip
Since $\mu$ is non-elementary, the limit set $\mathrm{Lim}(G_\mu) \subset S^1 = \partial_\infty \mathbb{H} $ of the semi-group $G_\mu$ generated by the support of $\mu$ is infinite. 
It follows that for any $\ell > 0$, $\mathrm{Lim}(G_\mu)$ contains a point $p$ such that the flat geodesic ray $\beta[x, p]$ has titled length of $\beta$ exceeds $3 \ell$. 
Parameterising $\beta$ by arc-length it follows that there exists a time $\beta_T$ such that the tilted lengths of the segments $[\beta_0, \beta_T]$, $[\beta_T, \beta_{2T}]$ and  $[\beta_{2T}, \beta_{3T}]$ all exceed $\ell$. 
To be precise, we may write $\beta$ as a concatenation such as above and then only the initial segment in the concatenation can be horizontal or vertical . 
In particular, $T$ could be taken to the flat length of this initial segment plus $\ell$. 
Once chosen, there is also a constant $\kappa > 0$ such that any slope $m_j$ along $[\beta_0, \beta_{3T}]$ that is not horizontal or vertical satisfies $1/\kappa < | m_j | < \kappa$. 
Except for (potentially) a horizontal or vertical prefix, the segment $\beta_{3T}$ satisfies the requirements of a good segment, that is, it contains three subsegments of length at least $\ell$ such that the absolute values of the slopes of the saddle connections along them are bounded away from $0$ and $\infty$ by $\kappa$.
In particular, the segment $[\beta_0, \beta_{3T}]$ can be used as a linear progress certificate for the metric $D$, where the progress achieved will depend on this bound $\kappa$.

\smallskip
We now consider $\shadow(\beta_{3T})$ and denote by $\partial_\infty \shadow(\beta_{3T})$ the limit set at infinity of the shadow. 
By definition, fixed points of hyperbolic isometries in the semigroup are dense in $\Lambda$ and since we are in the semigroup, we may assume that we can find the stable fixed point $p'$ of a (semi)-group element $g$ contained in the interior of $\partial_\infty \shadow(\beta_{3T})$. 
 
\smallskip
Let $\nu$ be the stationary measure for the random walk and let $I$ be a subset of $S^1$ such that $\nu(I) > 0$. 
Let $k \geqslant 0$ be the smallest integer such that $g^k I \subset \partial_\infty \shadow(\beta_{3T})$. 
By stationarity of $\nu$, 
\[
\nu (g^k I ) = \mu^{(k)}(g^k) \nu(g^{-k} g^kI ) + \sum_{h \neq g^k} \mu^{(k)} (h) \nu(h^{-1} I) 
\]
where $\mu^{(k)}$ is the $k$-fold convolution of $\mu$. 
Notice that the first term on the right is strictly positive because both $\mu^{(k)} (g^k)$ and $\nu(I)$ are strictly positive.
We deduce that $\nu(\partial_\infty \shadow(\beta_{3T})) \geqslant \nu(g^k I) > 0$. 
Denote $\nu(\partial_\infty \shadow(\beta_{3T}))$ by $\alpha$.

\smallskip
As discussed at the beginning of Section \ref{s.progress}, we now consider the set $\Omega$ of bi-infinite sample paths. 
By convergence to the boundary, almost every $\omega \in \Omega$ defines a bi-infinite hyperbolic geodesic $\gamma_\omega$ in $\mathbb{H}^2$.  
For $R > 0$, let $\Omega_R$ be the subset of those $\omega$ such that $d(\gamma_\omega, x) < R$. 
The subset $R$ is measurable and as $R \to \infty$, we have $(\nu \times \hat{\nu}) (\Omega_R) \to 1$. 
Hence, we may choose $R> 0$ such that $(\nu \times \hat{\nu} ) (\Omega_R) > 1 - \alpha/2$. 

\smallskip
Let $\Lambda$ be the subset of $\Omega_R$ of those $\omega = (w_n)$ such that $w_n x \to \partial \shadow(\beta_{3T})$. 
It follows that $(\nu \times \hat{\nu}) (\Lambda) > \alpha /2 $. 

\smallskip
We now consider the shift map $\sigma: \Omega \to \Omega$. 
Recall that for almost every bi-infinite sample path $\omega$, we get the tracked bi-infinite geodesic $\gamma_\omega$. 
Let $\gamma_\omega(j)$ is the point of $\gamma_\omega$ closest to $w_j x$.

\smallskip
By combining linear progress and sub-linear tracking in the metric $d$, namely \cite[Theorems 1.2 and 1.3]{Mah-Tio}, we deduce that for almost every $\omega = (w_n)$, the distance $d(\gamma_\omega(0), \gamma_\omega(n))$ where $\gamma_\omega(j)$ is the point of $\gamma_\omega$ closest to $w_j x$, grows linearly in $n$. 

\smallskip
By the ergodicity of $\sigma$, it follows that the asymptotic density of times $j$ such that $\sigma^j (\omega) \in \Lambda$ approaches $(\nu \times \hat{\nu} )(\Lambda)$, which exceeds $\alpha/2$; in particular, it is positive. 
Finally, the geodesic rays $\gamma_\omega$ and the ray from $x$ to the same point at infinite are positively asymptotic; that is up to the choice of an appropriate base-point, the distance between the corresponding points on the rays goes to zero.
\Cref{t.random} then follows by the same arguments as the proof of \Cref{t.linear-progress}. 

\end{proof} 

\section{Linear progress in analogous settings}

\begin{remark}\label{r.features}
We were very explicit about the constructions for fibered hyperbolic 3-manifolds but as the astute reader may have observed, the proofs rely on weaker features. 
We distill the essential features below.
\begin{enumerate}
\item A $\pi_1(\Sigma)$--equivariant assignment of shadows along any geodesic ray in the fibre with the property that if $0 < t$ is a sufficiently large time and $t< t'$, then $\shadow(\gamma_t) \subseteq \shadow(\gamma_{t'})$; 
\item a ladder-like construction in the total space with the property that if a shadow is contained in another shadow then its ladder is contained in the ladder of the other; 
\item the existence of finite segments in the fibre that achieve a specified nesting of ladders, that is, for any sufficiently large $R> 0$ there exists a segment $\beta$ of length $3\ell = O(R)$ and $\rho> 0$  such that 
\begin{itemize}
\item the ball $B(\beta_0, \rho) $ in the fibre is contained in the complement of $\shadow(\beta_\ell)$;
\item the ball $B(\beta_{3\ell}, \rho)$ is contained in $\shadow(\beta_{2\ell}) $; 
\item for any geodesic segment $[\beta'_0, \beta'_T]$ such that $\beta'_0 \in B(\beta_0, \rho)$ and $\beta'_T \in B(\beta_{3\ell}, \rho)$ the shadows satisfy $\shadow(\beta_\ell) \subseteq \shadow(\beta'_0)$ and $\shadow(\beta'_T) \subseteq \shadow(\beta_{2\ell})$; and 
\item $D_{\tot} (\partial \lad(\shadow(\beta_\ell)), \partial \lad(\shadow(\beta_{2\ell})) > R$;
\end{itemize}
\item for a good segment $\beta$ that satisfies (3) above, the set of geodesics in the fibre that pass through $B(\beta_0, \rho) $ and $B(\beta_{3\ell}, \rho)$ have a positive mass in the measure used for the sampling.
\end{enumerate}
Our proofs hold verbatim for fibrations (with surface/ surface group fibres) that exhibit these features establishing that a typical geodesic ray in the fibre makes linear progress in the metric on the total space.
\end{remark} 
\begin{figure}[h!]
  \centering
  \includegraphics[width=0.75\textwidth]{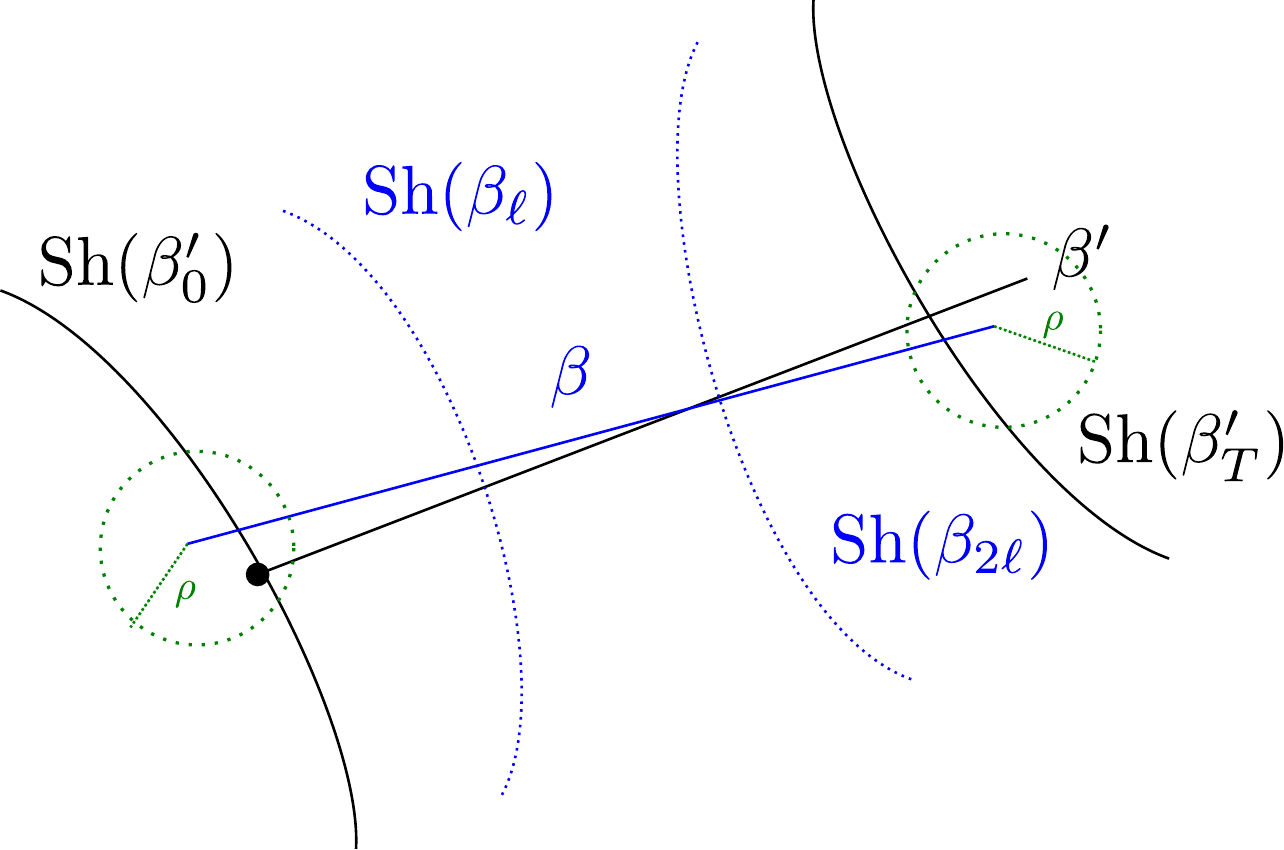}
  \caption{The general setup for shadows, as in
    Remark~\ref{r.features}~(3). $\beta$ is a good segment, whose
    endpoints are close to a segment $\beta'$. The nesting of ``inner'' shadows of $\beta, \beta'$, together with the separating of the ``outer'' shadows of $\beta, \beta'$ imply separation of shadows for $\beta'$, which is retained by its images in all fibres.}
  \label{fig:general-shadows}
\end{figure}

We will now give some explicit settings where these essential features hold and thus derive linear progress in the fibre.
We start with Gromov hyperbolic extensions of
surface groups, then consider canonical bundles over Teichm\"uller disks, and finally the
Birman exact sequence. 
In each case, different parts need to be adapted to check that the features in \Cref{r.features} hold but the general strategy remains the same.

\subsection{Hyperbolic extensions}
In this section, we discuss a finitely generated group extension
\[ 1 \to \pi_1(\Sigma) \to \Gamma \to Q \to 1, \] 
where $\Gamma$ is a hyperbolic group. 
We make no assumptions on $Q$, but remark
that in all known examples of this form, the group $Q$ is virtually free.
It is wide open if other examples exist, for instance, if there is a
hyperbolic extension of $\pi_1(\Sigma)$ by the fundamental group
$\pi_1(\Sigma')$ of another closed surface $\Sigma'$ with negative
Euler characteristic.

\smallskip 
We fix a finite generating set for $\Gamma$ that contains a
generating set for $\pi_1(\Sigma)$ and we equip $\Gamma$ with the
corresponding word metric.  We also choose, once and for all, an
equivariant quasi-isometry
\[ G: \mathbb{H}^2 \to \pi_1(\Sigma), \] and a basepoint
$x_0 \in \mathbb{H}^2$ with $G(x_0) = 1$.
We use it to identify the
kernel $\pi_1(\Sigma)$ of $p: \Gamma \to Q$ with the hyperbolic
plane. Such a choice is not unique (it is only up to quasi-isometry) -- but
we make the choice to use it later to sample geodesics
in $\mathbb{H}^2$. 

\smallskip 
By a result of Mosher, namely \cite[Theorem B]{Mos}, any
short exact sequence such as above with a hyperbolic, non-elementary kernel (like
$\pi_1(\Sigma)$) admits a quasi-isometric section
$\sigma:Q \to \Gamma$ with $\sigma(1) = 1$.

\smallskip
Fixing such a section, any $q\in Q$ induces a quasi-isometry
$\psi_q : \pi_1(\Sigma) \to \pi_1(\Sigma)$ by conjugation with
$\sigma(q)$, that is, $\psi_q(g) = \sigma(q)^{-1} g \, \sigma(q)$.

\smallskip
Suppose that $\beta$ is an infinite geodesic ray in $\pi_1(\Sigma)$ starting at the neutral element $1$. 
The closest point projection to $\beta$ is coarsely well-defined; that is, the image in $\beta$ of the set of closest points has bounded diameter.
For any point $r \in \beta$, let $N(\beta, r)$ denote the set of all points in $\pi_1(\Sigma)$ whose closest point
projection to $\beta$ (as a set) lies after $r$. 
The Gromov boundary of $\pi_1(\Sigma)$ is a circle, and the limit set $\partial_\infty N(\beta, r)$ is an interval. 
We then set our required shadow $\shadow (\beta, r)$ as the union of all bi-infinite geodesics in $\pi_1(\Sigma)$ whose both endpoints at infinity are contained in the limit set $\partial_\infty N(\beta, r)$. 
It is clear from the construction that if $r > r'$ then $\shadow (\beta, r) \subseteq \shadow (\beta, r')$ which is the containment property we require our assignment of shadows to satisfy in feature (1) of \Cref{r.features}. 
Furthermore, since $\pi_1(\Sigma)$ is quasi-isometric to the hyperbolic plane, observe that $\shadow (\beta, r)$ is a quasi-convex subset of $\pi_1(\Sigma)$ with a quasi-convexity constant independent of $r$. 

\smallskip
For any $q \in Q$, the corresponding quasi-isometry $\psi_q$ maps the interval $\partial_\infty N(\beta, r)$ to a possibly different interval $I_q(\beta, r) = (\psi_q)_\infty\partial_\infty N(\beta, r)$. 
Let $\shadow_q(\beta,r) $ be the union of bi-infinite geodesics in $\pi_1(\Sigma)$ whose both endpoints at infinity are contained in $I_q(\beta, r) $. 

\smallskip
Define 
\[ L(\beta,r) = \bigcup_{q \in Q} \sigma(q) \shadow_q(\beta, r) . \] 
This is our analogous ladder-like construction in this context. 
Using a slight extension of the methods of \cite{Mit}, we observe:
\begin{lemma}\label{lem:nesting}
  The set $L(\beta, r)$ is quasi-convex in $\Gamma$ and if $r'> r'$ then
  \[ L(\beta, r) \subseteq L(\beta, r') \]
\end{lemma}

\begin{proof}
We briefly indicate how the proof in \cite{Mit} needs to be adapted. 
The basic strategy is the same -- we define a Lipschitz projection $\Pi: \Gamma \to L(\beta, r)$. 
Since one can then project geodesics in $\Gamma$ to $L(\beta, r)$ without increasing their length too much, this will show un-distortion. 
Hyperbolicity of $\Gamma$ then implies quasi-convexity.

\smallskip	
As in \cite{Mit}, the definition of $\Pi$ involves the \emph{fibrewise} closest-point projection to the sets $\shadow_q(\beta, r)$. 
To show that the projection is Lipschitz, one needs to control the distance between $\Pi(x), \Pi(y)$ for points $x,y$ of distance $1$. 
There are two cases to consider. If $x,y$ are in the same fibre, the estimate stems from the fact that closest point
projections to quasi-convex sets in hyperbolic spaces are Lipschitz. 
If $x,y$ lie in adjacent fibres, then the estimate in \cite{Mit} relies on the fact that quasi-isometries coarsely commute with projections to geodesics in hyperbolic spaces. 
This fact is still true for projections to quasi-convex sets (with essentially the same proof), and so the argument extends. 
  
\smallskip  
Finally, we note that for $r < r'$, we have $\partial_\infty N(\beta, r) \subseteq \partial_\infty N(\beta, r')$ which implies that
$I_q(\beta, r) \subset I_q(\beta, r')$ for all $q \in Q$. 
This then implies $\shadow_q(\beta, r) \subseteq \shadow_q(\beta, r')$ and hence
\[ L(\beta, r) \subseteq L(\beta, r'). \]
  
\end{proof}

\Cref{lem:nesting} thus ensures that feature (2) in \Cref{r.features} holds.

\smallskip
We now construct appropriate \texttt{"}good\texttt{"} segments to achieve a specified nesting.
To start, we need to better understand what the shadows look like at infinity.
\begin{lemma}
  In the Gromov boundary $\partial_\infty \Gamma$, both $\partial_\infty L (\beta,r)$ and its
  complement have nonempty interior.
\end{lemma}
\begin{proof}
We first claim that there is a quasi-geodesic in $\Gamma$ whose endpoint lies in $\partial_\infty L (\beta, r)$. 
Namely, let $g \in \pi_1(\Sigma)$ be an element in the fibre so that the sequence $\beta'_n = g^n$ converges to a
point in $\partial_\infty N(\beta, r)$ as $n \to \infty$. 
Since the cyclic group generated by any infinite order element in a hyperbolic group is quasi-convex, hence undistorted, (compare e.g. \cite[III.F.3.10]{BH}), the sequence $\beta'_n$ is also a quasi-geodesic in $\Gamma$, thus proving the claim.
  
\smallskip  
We now claim that $\beta'_n$ nests in $L(\beta, r)$ as $n \to \infty$. 
That is, we claim that for any distance $R> 0$, we have $D_\Gamma (\beta'_n, \partial L (\beta, r)) > R$ for all sufficiently large $n$. 
Note that $d_{\pi_1(\Sigma)} (\beta'_n, \partial \shadow(\beta, r))$ becomes arbitrarily large as $n \to \infty$. 
We now choose a radius for a ball $B$ centred at identity in $Q$ such that $D_\Gamma(1, \sigma(q) ) > R$ for all $q \in Q - B$. 
We can then arrange $n$ to be sufficiently large so that the distance $D_\Gamma (\beta'_n, \partial \shadow_q (\beta, r) ) > R$ for all $q \in B$. 
The claim follows because $D_\Gamma (\beta'_n, \partial \shadow_q (\beta, r) ) > R$ for near-by fibres corresponding to $q \in B$, and $D_\Gamma(1, \sigma(q) ) $ is a lower bound on $D_\Gamma (\beta'_n, \partial \shadow_q (\beta, r) ) > R$ for all fibres corresponding to $Q - B$ and $D_\Gamma(1, \sigma(q) )  > R$ was arranged for these fibres.

\smallskip
By the claim above, we can choose $n$ sufficiently large to arrange that the ball in $\Gamma$ centred at $\beta'_n$ is contained deep in $L(\beta, r)$. 
It then follows that all geodesic rays in $\gamma$ starting at identity and passing through this ball converge to a point in the Gromov boundary that is contained in $\partial_\infty L(\beta, r)$. 
In particular, this means that $\partial_\infty L(\beta, r)$ has non-empty interior in $\partial_\infty \Gamma$. 

\smallskip
The claim for the complement $\partial_\infty \Gamma - \partial_\infty L(\beta,r)$ follows because the complement of $\shadow(\beta, r)$ in $\pi_1(\Sigma)$ contains the shadow $\shadow(\beta^-,r)$, where $\beta^-$ is the geodesic ray in $\pi_1(\Sigma)$ with its initial direction opposite to $\beta$. 

\end{proof}

\begin{corollary}
  There is an element $g_0$ of $\pi_1(\Sigma)$ whose axis has one endpoint
  in $\partial_\infty L(\beta,r)$ and one endpoint in the complement 
  of $\partial_\infty \Gamma - \partial_\infty L(\beta,r)$.
\end{corollary}
\begin{proof}
  One can either use an element as in the proof of the previous lemma,
  also assuming that $g^n$ converges to a point outside
  $\partial_\infty\shadow(\beta,r)$.

\smallskip
  Alternatively, choose open sets $U^+, U^-$ in
  $\partial_\infty L(\beta,r)$ and its complement. Since we have
  a continuous Cannon-Thurston map, there are
  intervals $I^+, I^-$ of the boundary $\partial_\infty \pi_1(\Sigma)$
  mapping (under this Cannon-Thurston map) into $U^+, U^-$.  We can
  choose an element $g_0$ of $\pi_1(\Sigma)$ whose axis endpoints are
  contained in $I^+, I^-$. This has the desired property.

\end{proof}

Recall that any infinite order element of a Gromov hyperbolic group
acts with north-south dynamics on the Gromov boundary.  
The element $g$ guaranteed by the previous corollary will act hyperbolically on
$\Gamma$ with axis endpoints in $U^+, U^-$.  
Thus, a large power $h = g^N$ has the property that it nests $L(\beta, r)$
properly into itself; in particular by choosing $N$ large enough, we
can guarantee that the distance between the boundaries of
$L(\beta, r)$ and $h L (\beta, r)$ is at least $R>0$ for any
choice of $R$. 
In other words, similar to $\ell$-good segments in the fibered case, 
feature (3) in \Cref{r.features} can be achieved by a geodesic segment $\beta$ in $\pi_1(\Sigma)$ 
from identity to a suitably high power $h = g^N$. 

\smallskip
Finally, since $\partial_\infty \shadow(\beta, r)$ is an interval with non-empty interior it has
positive measure with respect to geodesic sampling using the fixed quasi-isometry $G: \mathbb{H}^2 \to \pi_1(\Sigma)$.

\smallskip
Thus, feature (4) of \Cref{r.features} also holds and hence by replicating the proof of \Cref{t.random} we conclude that a typical ray in $\pi_1(\Sigma)$ makes linear progress in $\Gamma$. 

\subsection{Teichm\"uller disks}
In this section we consider the universal curve over a Teichm\"uller disk. 

\smallskip
A marked holomorphic quadratic differential on a closed Riemann surface defines charts to the complex plane via contour integration of a square root of the differential.
The transition functions are half-translations, that is of the form $z \to \pm z + c$. 
The natural action of $\mathrm{SL}_2(\mathbb{R})$ on $\mathbb{C} = \mathbb{R}^2$ preserves the form of the transitions and hence descends to an action on such differentials.

\smallskip
Let $q_0$ be such a quadratic differential on a Riemann surface $X_0$. 
The conformal structure is unchanged under the $\mathrm{SO}_2 (\mathbb{R})$-action on $q_0$ and hence the image of the orbit $\mathrm{SL}_2(\mathbb{R})q$ in Teichm\"{u}ller space is an isometrically embedded copy of $\mathbb{H}^2 = \mathrm{SL}_2 (\mathbb{R}) / \mathrm{SO}_2 (\mathbb{R})$.
This is called a \emph{Teichm\"{u}ller disk}.
We will use the notation $D_{q_0}$ for the disk and $X_0$ for the underlying marked Riemann surface for $q_0$. 
The canonical projection $\mathrm{SL}_2(\mathbb{R}) q_0 \to D_{q_0}$ exhibits the orbit as the unit tangent bundle of the Teichm\"uller disk. 
In particular, since $D_{q_0}$ is contractible the unit tangent bundle is trivial.

\smallskip
The universal curve over $D_{q_0}$ lifts to its universal cover to us gives the the bundle
\[
\mathbb{H}^2 \to E \to D_{q_0}.
\]

\smallskip
As discussed in e.g. \cite[Section 3.3-3.5]{DDLS}, the triviality of the unit tangent bundle allows us to equip the total space $E$ with a metric $d_E$ as follows.
\begin{itemize}
\item We choose a section $\sigma: D_{q_0}  \to \mathrm{SL}_2(\mathbb{R}) q_0$ such that $\sigma[q_0] = q_0$; 
\item We equip the fibre over $[q]$ with the lift of the singular flat metric on $\Sigma$ given by the differential $\sigma[q]$.  
\end{itemize}
In particular, the metrics on nearby fibres differ by (quasi-conformal) affine diffeomorphisms with bounded dilatation.

\smallskip
In fact, there is a convenient section of the unit tangent bundle. 
Namely, given any point $X \in D_q$, there is a (unique) Teichm\"uller extremal map from $X_0$ to $X$. 
The map pushes the quadratic differential $q_0$ on $X_0$ to a quadratic differential $q_X$ on $X$ -- which, if $X$ is defined by $Aq$ for $A\in\mathrm{SL}_2(\mathbb{R})$, differs from $Aq$ by a rotation.

\smallskip
We fix, once and for all, an identification of the fibre over the
base-point $q$ with the hyperbolic plane $\mathbb{H}^2$ up to
quasi-isometry. 
The goal is to discuss the behaviour of
a typical (hyperbolic) geodesic ray in that fibre for the metric $d_E$.

\smallskip 
One major difference from the preceding sections is that the
total space $E$ is no longer hyperbolic. 
This will mandate several adaptations from the previous two cases.

\smallskip 
We begin by defining shadows in the fibre as in the fibered 3-manifold case, that is, for a
flat geodesic ray $\beta$, we set $\shadow(\beta,r)$ as the shadow in the flat metric that contains the point at infinity $\partial_\infty \beta$ and has the optimal perpendicular $\beta^\perp(r)$ as the boundary.
As the hyperbolic and the singular flat metric are quasi-isometric, given a hyperbolic ray $\gamma$ in the fibre and a time $t$ along it, we can assign the shadow $\shadow(\beta, s(t))$, where $\beta$ is a flat ray fellow-travelling $\gamma$ for all time and $s(t)$ is a time along $\gamma$ assigned as in \Cref{e.assigned-time}. 
The assigned shadows then satisfy feature (1) in \Cref{r.features}.

\smallskip
We are now set up to carry out a ladder-like construction in this context.
Let $\beta$ be a flat geodesic. 
We define the ladder of $\beta$ by 
\[
\lad(\beta) = \bigcup_{X \in D_{q_0}} \sigma(X) (\beta) 
\]
where as above, $\sigma(X)$ is the affine diffeomorphism given by the trivialisation of the unit tangent bundle. 
Note that since $\sigma(X)$ is an affine map, it takes a flat geodesic in the fibre over $X_0$ to a flat geodesic in the fibre over $X$.

\smallskip
We then define 
 \[ L(\beta,r) = \bigcup_{X \in D_{q_0}}
\sigma(X) \shadow (\beta, r). \] 
It follows that along a flat geodesic ray $\beta$ the sets $L(\beta, r)$ satisfies feature (2) in \Cref{r.features}.

\smallskip
Arguing as in Lemma~\ref{lem:nesting} (observing that only the proof of quasiconvexity, and not of undistortion, relies on the hyperbolicity of the ambient space), we obtain the following.
\begin{lemma}\label{lem:TM-ladders}
	There is a constant $K$, so that for any flat geodesic $\beta$ in
	$\mathbb{H}^2$ and $r> 0$, the ladder $\lad(\beta)$ and the set $L(\beta, r)$ are $K$--undistorted in $E$ for the bundle metric
	$d_E$ defined above.
\end{lemma}

Our next aim is to show the existence of finite segments which produce any specified nesting on ladder-like sets. 
The additional complexity in this particular context is that a single saddle connection has no
lower bound on its flat length over all fibres.
To sidestep this issue, we finesse the definition of a good segment.

\smallskip
We consider the collection of flat geodesic segments that are a concatenation of segments of the form 
\[\beta =  i \ast v \ast h \ast f \]
for which
\begin{enumerate}
        \item the segments $v$ and $h$ are saddle connections with different slopes; in particular, we may assume that $v$ is almost vertical and $h$ almost horizontal; 
        \item the segments $i$ and $f$ have no singularities in their interior; and 
        \item the angles $\theta^L (i,v)$ and $\theta^L (h, f)$ subtended on the left along $\beta$ at the singularities satisfy $\theta^L (i, v) = \theta^L (h, f) = 3\pi/2$. 
\end{enumerate}

Note that condition (3) above implies that the angles $\theta^R (i, v)$ and $\theta^R (h,f)$ at the singularities satisfy $\theta^R (i, v) \geqslant 3\pi/2$ and $\theta^R (h, f) \geqslant 3 \pi/2$.
Let $v'$ and $v''$ be flat geodesic rays such that the concatenations $v' \ast v$ and $v'' \ast v$ are also flat geodesics and the angles $\theta^L (v', v) = \pi$ and $\theta^R (v', v) = \pi$. 
The concatenation $v' \ast v''$ is then a bi-infinite flat geodesic. 
Observe that the angle conditions imply that $(v' \ast v'') \cap i_\perp = \emptyset$. 
Similarly let $h'$ and $h''$ be flat rays such that the concatenations $h \ast h' $ and $h \ast h''$ are also flat geodesics and the angles $\theta^L (h \ast h') = \theta^R (h \ast h'') = \pi$.
Then the concatenation $h' \ast h''$ is a bi-infinite flat geodesic and by the same logic regarding angles $(h' \ast h'') \cap f_\perp = \emptyset$. 

\smallskip
We now derive
\begin{lemma}\label{l.contains}
A flat geodesic segment with endpoints on $i_\perp$ and $f_\perp$ contains $v \ast h$. 
\end{lemma} 

\begin{proof}
Let $\beta'$ be a flat geodesic segment with endpoints $x_i$ on $i_\perp$ and $x_f$ on $f_\perp$. 
Since $v' \ast v''$ separates $x_i$ and $x_f$, the segment $\beta'$ must intersect $v' \ast v''$. 
Similarly $\beta'$ must also intersect $h' \ast h''$. 
Breaking symmetry, suppose that $\beta'$ intersects $v'$ and $h'$ in points $p$ and $q$ respectively.
The concatenation $v' \ast v \ast h \ast h'$ also gives a geodesic segment between $p$ and $q$. 
This implies that $\beta'$ and $v' \ast v \ast h \ast h'$ must coincide between $p$ and $q$ and thus $\beta'$ contains $v \ast h$.
Identical arguments apply for possibilities for intersections of $\beta'$ with $v' \ast v''$ and $h' \ast h''$ which concludes the proof of the lemma. 

\end{proof} 

We now parameterise $\beta$ and denote by times $s < t$ the midpoints of $i$ and $f$ along $\beta$.
By \Cref{l.contains}, any flat geodesic from $\mathbb{H}^2 - \shadow(\beta, s)$ to $\shadow(\beta, t)$ contains $v \ast h$. 

\smallskip
As the maps $\sigma(X)$ act by affine diffeomorphisms, it immediately follows that
\begin{corollary}\label{c.contains}
A flat geodesic segment with endpoints on $\sigma(X)(i_\perp)$ and $\sigma(X)(f_\perp)$ contains the segment $\sigma(X) (v \ast h)$. 
\end{corollary}

Given any $\ell > 0$ it is obvious from the asymptotics of saddle connections that we can choose a segment $\beta$ such that both $v$ and $h$ have flat length at least $\ell$.
We will call such a segment $\beta$ to be $\ell$-good. 

\smallskip
With the assumption that $v$ is almost very vertical and $h$ almost horizontal, notice that $\sigma(X) (v \ast h)$ is bounded below by $O(\ell)$ for any $X$ in $D_{q_0}$. 

\smallskip
The main point now is that
\begin{lemma} 
Given $R > 0$ there exists $\ell> 0$ such that for any $\ell$-good segment $\beta = i \ast v \ast h \ast f$ and times $s< t$ corresponding to midpoints of $i$ and $f$ along $\beta$, any geodesic segment in $E$ that joins $E - L(\beta, s)$ to $L(\beta, t)$ has length at least $R$. 
\end{lemma} 

\begin{proof} 
Since the sets $L(\beta, s)$ and $L(\beta, t)$ are undistorted by \Cref{lem:TM-ladders}, the $E$-distance
between $\partial L(\beta, s)$ and $\partial L(\beta, t)$ can be (coarsely) computed in $L(\beta, s)$. 
So suppose that a geodesic segment in $L(\beta, s)$ coarsely gives the distance between  $\partial L(\beta, s)$ and $\partial L(\beta, t)$. 
Let $x_i \in i_\perp = \partial \shadow(\beta, s)$ and $x_f \in f_\perp= \partial \shadow(\beta, t)$ be points such that the geodesic segment above has its endpoints in the ladders $\lad(x_i)$ and $\lad(x_f)$. 
Since ladders are also undistorted by \Cref{lem:TM-ladders}, it suffices to show that the distance in $E$ between any point of $\lad(x_i)$ and $\lad(x_f)$ is uniformly bounded from below by a constant that is linear in $\ell$. 
This follows from \Cref{c.contains} and the observation preceding the lemma that $\sigma(X) (v \ast h)$ is bounded below by $O(\ell)$ for any $X$ in $D_{q_0}$. 
This means that given $R> 0$, we can indeed find $\ell> 0$ that achieves the nesting as required. 

\end{proof}

\subsection{Point-pushing groups}
The final setting that we consider is given by the Birman exact sequence:
\[ 1 \to \pi_1(S_g,p) \to \mathrm{Mod}(S_g,p) \to \mathrm{Mod}(S_g) \to 1. \]
We follow essentially the same strategy as in the previous sections.
As in the section on hyperbolic extensions, we construct nested shadows along geodesic rays in the fibre in exactly the same way.
The construction of ladders and subsequently the sets $L(\beta, r)$ is also identical.

\smallskip
The definition of segments that achieve nesting as in point (3) of \Cref{r.features} require care.
Roughly speaking given $\ell > 0$ sufficiently large, we need a segment such that all $\text{Mod}(S_g)$ images of it have lengths bounded below by $O(\ell)$.

\smallskip
We first recall the following basic fact from elementary hyperbolic geometry.
\begin{lemma}\label{l.self}
Let $S$ be a closed surface. 
Given any complete hyperbolic structure $X$ on $S$ and any $R > 0$, there exists a natural number $k = k(S,R)$ such that any geodesic arc on $X$ with at least $k$ self-intersections has length at least $R$.
\end{lemma}

We now fix a complete hyperbolic structure $X$ on $S_g$.
Given $R>0$ sufficiently large let $k$ be as in \Cref{l.self}. 
Fix a geodesic arc on $X$ with $k$ self-intersections. 
Let $\gamma$ be a bi-infinite lift in $\mathbb{H}^2$ of this arc on $X$.
We parameterise $\gamma$ with unit speed.
Given a time $s$ let $\alpha_s$ be the bi-infinite geodesic orthogonal to $\gamma$ at $\gamma_s$. 
We let $H^-_s$ and $H^+_s$ be the half-spaces with boundary $\alpha_s$ such that $\gamma_t$ converges in to $H^-_s$ as $t \to -\infty$ and $\gamma_t$ converges in to $H^+_s$ as $t \to \infty$. 

\smallskip
We may then choose a time $s$ sufficiently large such that any geodesic segment $\gamma'$ with endpoints in $H^-_0= H^-_{-s}$ and $H^+_0 = H^+_s$ fellow-travels a long enough arc of $\gamma$ to ensure that the projection of $\gamma'$ to $X$ has at least $k$ self-intersections. 
In fact, by passing to a larger $s$ if required, we can ensure that there are at least $k$ group elements $g_i \, : 1 \leqslant i \leqslant k$ such that the half-spaces in the list $\{H^-_0, H^+_0, g_1 H^-_0, g_1 H^+_0, \cdots, g_k H^-_0, g_k H^+_0 \}$ are pairwise well-separated and the pairs $(g_i H^-_0, g_i H^+_0)$ all link the pair $(H^-_0, H^+_0)$. 
Let $I^-_0 = \partial_\infty H^-_0$ and $I^+_0 = \partial_\infty H^+_0$ be the pair of intervals at infinity.
Similarly, we get the pairs of intervals $I^-_i = \partial_\infty g_i H^-_0$ and $I^+_i = \partial_\infty g_i H^+_0$.
The separation of half-spaces implies that these intervals are all pairwise disjoint and along the circle $S^1 = \partial_\infty \mathbb{H}^2$ the pairs $(I^-_i, I^+_i)$ are all linked with the pair $(I^-_0, I^+_0)$. 

\smallskip
The hyperbolic structure $X$ defines a group-equivariant quasi-isometry $\psi: \mathbb{H}^2 \to \pi_1(S_g)$ that gives a homeomorphism $\psi_\infty$ of their Gromov boundaries.
So, we get intervals $J^-_i = \psi_\infty (I^-_i)$ and $J^+_i = \psi_\infty (I^+_i)$ in the Gromov boundary of $\pi_1(S_g)$.
Being a homeomorphism, $\psi_\infty$ preserves the linking and hence the pairs $(J^-_i, J^+_i)$ all link $(J^-_0, J^+_0)$. 

\smallskip
Realising a mapping class on $S_g$ as an actual automorphism $f$ of $\pi_1(S_g)$, the action of $f$ on $\pi_1(S_g)$ extends to a homeomorphism $f_\infty$ of the Gromov boundary.
Hence, the pairs $(f_\infty(J^-_i), f_\infty (J^+_i))$ continue to link the pair $(f_\infty J^-_0, f_\infty J^+_0)$. 
Let $H^-_{f,i}$ and $H^+_{f,i}$ be the half-spaces in $\mathbb{H}^2$ such that $\partial_\infty H^-_{f,i} = f_\infty(J^-_i)$ and $\partial_\infty H^+_{f,i} = f_\infty (J^+_i)$. 

\smallskip
We note that
\begin{lemma}
For any mapping class on $S_g$ and any automorphism $f$ of $\pi_1(S_g)$ in its class, 
\[
d_{\hyp} ( H^-_{f,0} \, , H^+_{f,0} ) \geqslant R.
\]
\end{lemma}

\begin{proof}
Because of the linking, any geodesic segment that connects $H^-_{f,0}$ to $H^+_{f,0}$ projects to an arc on $X$ that self-intersects at least $k$ times. 
By \Cref{l.self}, the arc has length at least $R$ and the lemma follows.
\end{proof} 

By construction, $\cup_{f \in \text{Mod}(S_g)} H^+_{f, 0} = L(\gamma, s) $ and $\cup_{f \in \text{Mod}(S_g) } H^-_{f, 0}  = \text{Mod}(S_g, p) - L(\gamma, -s)$. 

\smallskip
We now define the nesting segment to be the segment $\beta = [\gamma_{-3s}, \gamma_{3s}]$. 
To compare constants with feature (3) in \Cref{r.features}, we set $\ell = 2s$ and reset time zero to be $-3s$. 
We then get a parameterised segment $\beta$ of length $3\ell$ for which the ball $B(\beta_0, \rho)$ is contained in the complement of $L(\beta, \ell)$,  the ball $B(\beta_{3\ell}, \rho)$ is contained in $L(\beta, 2\ell)$ and 
$D_{\text{Mod}(S_g, p)} (\text{Mod}(S_g, p) - L(\beta, \ell), L(\beta, 2\ell)) > R$, as required.

%%%%%%%%%%%%%%%%%%%%%%%%%%%%%%%%%%%%%%%%%%%%%%%%%%

\end{document}